\documentclass{amsart}

\usepackage[english]{babel}
\usepackage{hyperref}

\usepackage[utf8]{inputenc}
\usepackage[T1]{fontenc}
\usepackage{pifont}

\usepackage[margin=3cm]{geometry}

\usepackage{amsmath}
\usepackage{amssymb}
\usepackage{amsthm}
\usepackage{amscd}
\usepackage{mathrsfs}
\usepackage[abbrev,backrefs]{amsrefs}

\usepackage{tikz-cd}

\usepackage{graphicx}

\numberwithin{equation}{section}

\newtheorem{theorem}{Theorem}[section]	
\newtheorem{proposition}[theorem]{Proposition}	
\newtheorem{corollary}[theorem]{Corollary}	
\newtheorem{lemma}[theorem]{Lemma}
\newtheorem{definition}[theorem]{Definition}		
\newtheorem{example}[theorem]{Example}

\newtheorem{remark}[theorem]{Remark}

\newcommand{\Z}{\mathbb{Z}}

\newcommand{\N}{\mathbb{N}}
\newcommand{\C}{\mathscr{C}}
\newcommand{\T}{\mathscr{T}}
\newcommand{\F}{\mathscr{F}}
\newcommand{\z}{\mathscr{Z}}
\newcommand{\n}{\mathscr{N}}

\newcommand{\E}{\mathscr{E}}

\newcommand{\M}{\mathscr{M}}

\begin{document}

\title{\textsc{Torsion theories and coverings of preordered groups}}
\author{Marino Gran*}
\thanks{marino.gran@uclouvain.be * \,}
\author{Aline Michel}
 \email{aline.michel@uclouvain.be}
\thanks{The second author's research is funded by a FRIA doctoral grant of the \emph{Communaut\'e française de Belgique}}

\address{Universit{\'e} Catholique de Louvain, Institut de Recherche en Math{\'e}matique et Physique, Chemin du Cyclotron 2, 1348 Louvain-la-Neuve, Belgium}
\date{}
\begin{abstract}
 In this article we explore a non-abelian torsion theory in the category of preordered groups: the objects of its torsion-free subcategory are the partially ordered groups, whereas the objects of the torsion subcategory are groups (with the total order). The reflector from the category of preordered groups to this torsion-free subcategory has stable units, and we prove that it induces a monotone-light factorization system. We describe the coverings relative to the Galois structure naturally associated with this reflector, and explain how these coverings can be classified as internal actions of a Galois groupoid. Finally, we prove that in the category of preordered groups there is also a pretorsion theory, whose torsion subcategory can be identified with a category of internal groups. This latter is precisely the subcategory of protomodular objects in the category of preordered groups, as recently discovered by Clementino, Martins-Ferreira, and Montoli.
\end{abstract}

\maketitle

\section{Introduction}
The category $\mathsf{PreOrdGrp}$ of \emph{preordered groups} is the category whose objects $(G,\le)$ are groups $G$ endowed with a  preorder relation $\le$ on $G$ which is compatible with the group structure $+$:  $a \le c$ and $b \le d$ implies $a + b \le c + d$, for all $a,b,c,d \in G$. The morphisms in this category are preorder preserving group morphisms. 

Alternatively, a preordered group $(G,\le)$ can be seen in a different way. Indeed, consider the submonoid
\[P_{G} = \{g \in G \ \arrowvert \ 0 \le g \}\]
of $G$, called the \textit{positive cone} of $G$, that has the property of being closed under conjugation in $G$. It is well-known that the category of preordered groups is isomorphic to the category whose objects are pairs $(G,P_{G})$, where $G$ is a group and $P_{G}$ is its positive cone, and whose arrows $(f,\bar{f}) : (G,P_{G}) \rightarrow (H,P_{H})$ are pairs $(f,\bar{f})$ where $f : G \rightarrow H$ is a group morphism and $\bar{f} : P_{G} \rightarrow P_{H}$ is a monoid morphism, such that the following diagram commutes (the vertical morphisms are the inclusions):
\begin{equation}\label{morphismPreOrd}
\begin{tikzcd}
P_{G} \arrow[r,"\bar{f}"] \arrow[d,tail]
& P_{H} \arrow[d,tail]\\
G \arrow[r,"f"']
& H.
\end{tikzcd}
\end{equation}
In this article we will always work with this latter equivalent presentation of the category $\mathsf{PreOrdGrp}$.
 In \cite{Clementino-Martins-Ferreira-Montoli} Clementino, Martins-Ferreira and Montoli proved that $\mathsf{PreOrdGrp}$ has some remarkable exactness properties. First of all, $\mathsf{PreOrdGrp}$ is a \emph{normal} category \cite{JZ}: this means that $\mathsf{PreOrdGrp}$ has a zero-object, any arrow in it can be factorized as a normal epimorphism (i.e. a cokernel) followed by a monomorphism, and these factorizations are pullback-stable. Secondly, in this category normal epimorphisms and effective descent morphisms coincide, an observation which is fundamental in our study of the coverings in $\mathsf{PreOrdGrp}$.

Our first result is that $\mathsf{PreOrdGrp}$ contains two full (replete) subcategories, denoted by $\mathsf{Grp}$ and $\mathsf{ParOrdGrp}$, which form a (non-abelian) torsion theory $(\mathsf{Grp}, \mathsf{ParOrdGrp})$ (Proposition \ref{torsion OrdGrp}). Here the objects of the torsion subcategory $\mathsf{Grp}$ are those preordered groups $(G,G)$ such that the positive cone $P_G$ is $G$ itself, whereas the objects in the torsion-free subcategory $\mathsf{ParOrdGrp}$ have the property that the positive cone is a \emph{reduced monoid}: $x + y = 0$ implies $x = y = 0$, for any $x,y \in P_G$.
Via the isomorphism of categories recalled above, preordered groups with a reduced monoid as positive cone exactly correspond to \emph{partially ordered groups}. We then have a reflective subcategory 
\begin{equation}\label{main-adj-2}
\begin{tikzcd}[column sep = tiny]
\mathsf{PreOrdGrp} \arrow[rr,shift left=0.2cm,"F"]
& \bot
& \mathsf{ParOrdGrp}, \arrow[ll,shift left=0.2cm,hook',"U"]
\end{tikzcd}
\end{equation}
where each component of the unit of the adjunction is a normal epimorphism.
We prove that the reflector $F \colon \mathsf{PreOrdGrp} \rightarrow \mathsf{ParOrdGrp}$ has stable units \cite{Cassidy-Hebert-Kelly} in Proposition \ref{Stable-units}, and this implies that the adjunction can be studied from the point of view of Categorical Galois Theory \cite{Janelidze}. By constructing, for any preordered group $(G,P_{G})$, an effective descent morphism whose domain is a partially ordered group and whose codomain is $(G,P_{G})$ (Proposition \ref{ParOrdGrp covers PreOrdGrp}), we can show that this adjunction induces a \emph{monotone-light} factorization system $({\mathcal E}', \M^*)$ (Theorem \ref{description-coverings}). The class ${\mathcal E}'$ consists of the morphisms in $\mathsf{PreOrdGrp}$ which are \emph{stably} in $\mathcal E$, this meaning that the pullback of a morphism in ${\mathcal E}'$ along any arrow is in $\mathcal E$, i.e. it is inverted by the reflector $F \colon \mathsf{PreOrdGrp} \rightarrow \mathsf{ParOrdGrp}$.

 The class ${\M}^*$ is the important class of \emph{coverings}, in the sense of Galois theory, with respect to the adjunction \eqref{main-adj-2}. In elementary terms, the coverings turn out to be the morphisms $(f,\bar{f}) : (G,P_{G}) \rightarrow (H,P_{H})$ as in \eqref{morphismPreOrd} having a partially ordered kernel: $\mathsf{Ker}(f,\overline{f}) \in \mathsf{ParOrdGrp}$.
In the fourth section we then compare our results with the ones on \emph{locally semisimple coverings} from \cite{JMT}. Categorical Galois Theory \cite{Janelidze} then provides a classification theorem of the coverings in $\mathsf{PreOrdGrp}$ in terms  of the Galois groupoid of the effective descent morphism mentioned above. In our context this groupoid is actually an equivalence relation, 
and the above-mentioned description of the coverings in terms of actions (i.e. discrete fibrations) is explicitly given in Theorem \ref{thm action}). 

It turns out that the adjunction \eqref{main-adj-2} also induces a \emph{pretorsion theory}  (in the sense of \cite{FF, Facchini-Finocchiaro-Gran}) in $\mathsf{PreOrdGrp}$. This is given by the pair $(\mathsf{ProtoPreOrdGrp}, \mathsf{ParOrdGrp})$, where the torsion part is this time the category $\mathsf{ProtoPreOrdGrp}$ whose objects $(G,P_{G})$ are characterized by the fact that the positive cone $P_G$ is a \emph{group}. As shown in \cite{Clementino-Martins-Ferreira-Montoli} these objects are precisely the so-called \emph{protomodular objects} \cite{MRV} of $\mathsf{PreOrdGrp}$. This interesting category, which can be also seen as the category of internal groups in the category $\mathsf{PreOrd}$ of preordered sets (see \cite{Clementino-Martins-Ferreira-Montoli}),  is not only coreflective (as any torsion subcategory of a pretorsion theory is) but also reflective in $\mathsf{PreOrdGrp}$: this is proved in Proposition \ref{reflective}, where an explicit description of the reflector is provided.

We conclude this introduction by mentioning the related work in \cite{FFG, Xarez}, where similar results have been obtained in the context of internal preorders in an exact category. The results on preordered groups presented in this article are not special cases of the ones presented in those references, since a preordered group is not an internal preorder in the category $\mathsf{Grp}$ of groups.  

\vspace{3mm}

\noindent {\bf Acknowledgement}. The authors are grateful to the anonymous referee for some very useful suggestions on a preliminary version of the article.

\section{Preliminaries}
\subsection*{Torsion theories in normal categories}

In this part we briefly recall the notion of torsion theory in a normal category. There are several approaches to non-abelian torsion theories in various contexts, which can be found, for instance, in \cite{Bourn-Gran,CDT, Everaert-Gran, JT} (and in the references therein). 

 A finitely complete category $\C$ is \emph{normal} \cite{JZ} if 
 \begin{enumerate}
 \item $\C$ has a zero object, denoted by $0$;
 \item any arrow $f \colon A \rightarrow B$ in $\C$ factors as a normal epimorphism (i.e. a cokernel) followed by a monomorphism;
 \item normal epimorphisms are stable under pullbacks: in a pullback diagram
  \begin{equation}\label{pullback}
\begin{tikzcd}
E \times_{B} A \arrow[r,"\pi_{2}"] \arrow[d,"\pi_{1}"']
& A \arrow[d,"f"]\\
E \arrow[r,"p"']
& B
\end{tikzcd}
\end{equation}
$\pi_2$ is a normal epimorphism whenever $p$ is a normal epimorphism.
 \end{enumerate} 
Many familiar algebraic categories, such as groups, abelian groups, rings, Lie algebras, crossed modules of groups and of Lie algebras, are normal. For a variety whose theory has a  unique constant $0$, being normal is equivalent to being
$0$-regular (in the sense of \cite{Fichtner}): each congruence is determined by the equivalence class of $0$.
Any semi-abelian category is normal, as well as any homological category \cite{BB}. The categories of topological groups \cite{BC}, compact groups, Heyting semi-lattices \cite{Joh} and cocommutative Hopf algebras over a field \cite{GSV} are all examples of normal categories. It was recently proved that the category of preordered groups is also normal \cite{Clementino-Martins-Ferreira-Montoli}, and this observation will be important for our work. 

In a normal category there is a natural notion of \emph{short exact sequence}: two composable arrows $\kappa$ and $f$ form a short exact sequence 
$$\begin{tikzcd}
0 \arrow[r]
& A \arrow[r,"\kappa"] 
& B \arrow[r,"f"] 
& C  \arrow[r]
& 0
\end{tikzcd}$$
if $\kappa = \mathsf{ker}(f)$ and $f = \mathsf{coker(\kappa})$.
Two useful properties of normal categories are the following (see \cite{BJ}):
\begin{lemma} \label{property-normal}
Let $\C$ be a normal category.
\begin{enumerate}
\item A morphism $f \colon A \rightarrow B$ in $\C$ is a monomorphism if and only if its kernel $\mathsf{Ker}(f)$ is trivial: $\mathsf{Ker}(f) \cong 0$.
\item Given a commutative diagram of short exact sequences in $\C$
\begin{equation}\label{generic-ses}
\begin{tikzcd}
0 \arrow[r]
& A \arrow[r,"\kappa"] \arrow[d,,"a"']
& B \arrow[r,"f"] \arrow[d,"b"]
& C \arrow[d,"c"] \arrow[r]
& 0\\
0 \arrow[r]
& A' \arrow[r,"\kappa'"']
& B' \arrow[r,"f'"']
& C' \arrow[r]
& 0
\end{tikzcd}
\end{equation}
the left-hand square is a pullback if and only if the arrow $c$ is a monomorphism.
\end{enumerate}
\end{lemma}

\begin{definition} \label{def torsion theory}
A \emph{torsion theory} in a normal category $\C$ is given by a pair $(\T,\F)$ of full (replete) subcategories of $\C$ such that:
\begin{itemize}
\item[$(a)$] the only arrow from any $T \in \T$ to any $F \in \F$ is the zero arrow;
\item[$(b)$] for any object $C$ of $\C$ there exists a short exact sequence
\begin{center}
\begin{tikzcd}
0 \arrow[r]
& T \arrow[r,"\epsilon_{C}"]
& C \arrow[r,"\eta_{C}"]
& F \arrow[r]
& 0
\end{tikzcd}
\end{center}
whith $T \in \T$ and $F \in \F$.
\end{itemize}
\end{definition}

Given a torsion theory $(\T,\F)$ in a normal category $\C$ the subcategory $\T$ is called a \textit{torsion subcategory} of $\C$ and the subcategory $\F$ a \textit{torsion-free subcategory} of $\C$, by analogy with the terminology used for the classical torsion theory $(\mathsf{Ab}_{t.}, \mathsf{Ab}_{t.f.})$ in the category $\mathsf{Ab}$ of abelian groups, where $\mathsf{Ab}_{t.}$ is the category of torsion abelian groups and $\mathsf{Ab}_{t.f.}$ the category of torsion-free abelian groups.

Observe that the exact sequence in Definition \ref{def torsion theory} ($b$) is unique, up to isomorphism. Indeed, 
assume that for an object $C$ in $\C$ we have two short exact sequences, with kernel in $\T$ and cokernel in $\F$:
\begin{equation} \label{uniqueness exact sequence}
\begin{tikzcd}
0 \arrow[r]
& T \arrow[r,"\epsilon_{C}"] \arrow[d,dotted,"t"']
& C \arrow[r,"\eta_{C}"] \arrow[d,equal, "1_C"]
& F \arrow[d,dotted,"f"] \arrow[r]
& 0\\
0 \arrow[r]
& T' \arrow[r,"\epsilon_{C}'"']
& C \arrow[r,"\eta_{C}'"']
& F' \arrow[r]
& 0.
\end{tikzcd}
\end{equation}
Then, since $\eta_C$ is the cokernel of $\epsilon_C$ and $\eta_C' \cdot \epsilon_{C}$ is the zero arrow (by $(a)$), there exists a unique morphism $f \colon F \rightarrow F'$ such that $f \cdot \eta_C = \eta_{C}'$. It is then easy to show that $f$ is an isomorphism (its inverse is induced, symmetrically, by the universal property of the cokernel $\eta_C'$). Dually, by using the universal property of kernels, there is  an isomorphism $t : T \rightarrow T'$.

Now, consider a morphism $\phi : C \rightarrow C'$ in $\C$ as in the following diagram
\begin{equation} \label{diagram exact sequence}
\begin{tikzcd}
0 \arrow[r]
& T \arrow[r,"\epsilon_{C}"] \arrow[d,dotted,"T(\phi)"']
& C \arrow[r,"\eta_{C}"] \arrow[d,"\phi"]
& F \arrow[d,dotted,"F(\phi)"] \arrow[r]
& 0\\
0 \arrow[r]
& T' \arrow[r,"\epsilon_{C'}"']
& C' \arrow[r,"\eta_{C'}"']
& F' \arrow[r]
& 0
\end{tikzcd}
\end{equation}
where the two rows are the unique short exact sequences in Definition \ref{def torsion theory} $(b)$ associated with $C$ and $C'$, respectively. As above the dotted arrows $F(\phi) : F \rightarrow F'$ and $T(\phi)  \colon T \rightarrow T'$ are induced by the universal properties of the cokernel $\eta_C$ and of the 
kernel $\epsilon_{C'}$, respectively.
This construction then gives rise to two functors, $T : \C \rightarrow \T$ and $F : \C \rightarrow \F$, which are the right (respectively, the left) adjoint of the inclusion functor $V \colon \T \rightarrow \C$ (respectively, $U \colon \F \rightarrow \C$) (see \cite{Bourn-Gran}, for instance). \
The functor $F : \C \rightarrow \F$ is a (normal epi)-reflector, i.e. a reflector with the property that each component $\eta_{C} : C \rightarrow UF(C)$ of the unit $\eta$ of the adjunction $F \dashv U$ is a normal epimorphism. The dual statement is also true: the torsion subcategory $\T$ is (regular mono)-coreflective in $\C$. Note that the morphism $\eta_C \colon C \rightarrow UF(C)$ is the arrow $\eta_C \colon C \rightarrow F$ in Definition \ref{def torsion theory} above. Similarly the $C$-component of the counit of the adjunction $V \dashv T$ is the arrow $\epsilon_{C} : (T =) VT(C) \rightarrow C$ of Definition \ref{def torsion theory}.

\subsection*{Effective descent morphisms}
The notion of \emph{effective descent morphism} can be defined in terms of \emph{discrete fibrations} of \emph{internal equivalence relations}, two concepts that we are now going to recall. For more details on the content of this section the interested reader can refer to \cite{Borceux,JST,JK}. Let $\C$ be any category with pullbacks. An \emph{internal equivalence relation} is a diagram 
\begin{equation} \label{diag equivalence relation}
\begin{tikzcd}
R \times_X R \arrow[r,shift left=2.3,"p_{1}"] \arrow[r,shift right=2.3,"p_{2}"'] \arrow[r,"\tau" description]
& R \arrow[r,shift left=2.3,"r_1"] \arrow[r,shift right=2.3,"r_2"']  \arrow[out=125,in=55,loop,looseness=4,"\sigma"]
& X \arrow[l,"\Delta"' description]
\end{tikzcd}
\end{equation}
in $\C$,
where $(R \times_X R, p_1, p_2)$ is defined by the following pullback
\begin{center}
\begin{tikzcd}
R \times_{X} R \arrow[r,"p_{2}"] \arrow[d,"p_{1}"']
& R \arrow[d,"r_1"]\\
R \arrow[r,"r_2"']
& X,
\end{tikzcd}
\end{center}
the morphisms $r_1$ and $r_2$ are jointly monomorphic, and the following identities are satisfied:
\begin{enumerate}
\item $r_1 \cdot \Delta = 1_{X} = r_2 \cdot \Delta$ (reflexivity);
\item $r_1\cdot \sigma = r_2$, and $r_2 \cdot \sigma = r_1$ (symmetry);
\item $r_1 \cdot p_{1} = r_1 \cdot \tau$ and $r_2 \cdot p_2 = r_2 \cdot \tau$ (transitivity).
\end{enumerate}
Of course, when $\C$ is the category $\mathsf{Set}$ of sets and functions, 
$r_1$ and $r_2$ are the first and second projections of the relation $R$, $\Delta$ is the ``diagonal map'' yielding the reflexivity of the relation, $\sigma$ and $\tau$ are the ``symmetry” and the “transitivity” maps, respectively.
In other words, an internal equivalence relation in $\mathsf{Set}$ is just an equivalence relation in the usual sense. More generally, an internal equivalence relation in any variety $\mathbb V$ of universal algebras is a \emph{congruence} \cite{BS}, i.e. an equivalence relation which is also compatible with the operations of the algebraic theory of $\mathbb V$. 
\begin{example}
The kernel pair $(Eq(p),p_1,p_2)$ of a morphism $p \colon E \rightarrow B$  is always an internal equivalence relation in $\C$. Note that, in universal algebra, the kernel pair of a homomorphism is sometimes called its ``kernel'' \cite{BS}. 
\end{example}
From now on, to simplify the notations, an internal equivalence relation in a category $\C$ as in \eqref{diag equivalence relation} will be depicted as follows:
\begin{center}
\begin{tikzcd}
R \arrow[r,shift left,"r_1"] \arrow[r,shift right,"r_2"']
& X.
\end{tikzcd}
\end{center}
A \emph{discrete fibration of internal equivalence relations} from $(R,r_1,r_2)$ to $(R',r'_1,r'_2)$ is given by a couple $(f_0,f_1)$ of arrows in $\C$ such that all the corresponding squares in the diagram
\begin{equation} \label{def discrete fibration diag1}
\begin{tikzcd}
& R \arrow[r,shift left,"r_1"] \arrow[r,shift right,"r_2"']  \arrow[d,"f_1"']
& X  \arrow[d,"f_0"]\\
& R' \arrow[r,shift left,"r'_1"] \arrow[r,shift right, "r'_2"'] 
& X' 
\end{tikzcd}
\end{equation}
commute and such that the diagram
\begin{equation} \label{def discrete fibration diag2}
\begin{tikzcd}
& R \arrow[r,"r_2"]  \arrow[d,"f_1"']
& X \arrow[d,"f_0"]\\
& R' \arrow[r,"r'_2"'] 
& X' 
\end{tikzcd}
\end{equation}
is a pullback.

\begin{remark}
\emph{Note that, by the the symmetry of the relations, the conditions above imply that also the following diagram is a pullback:}
\begin{center}
\begin{tikzcd}
& R \arrow[r,"r_1"]  \arrow[d,"f_1"']
& X \arrow[d,"f_0"]\\
& R' \arrow[r,"r'_1"'] 
& {X'.}
\end{tikzcd}
\end{center}
\end{remark} 
Given a morphism $p \colon E \rightarrow B$, the discrete fibrations of equivalence relations with codomain $Eq(p)$ 
\begin{equation} \label{discrete}
\begin{tikzcd}
R \arrow[r,shift left,"r_1"] \arrow[r,shift right, "r_2"'] \arrow[d, "f_1"']
& F \arrow[d, "f_0"]\\
Eq(p) \arrow[r,shift left,"p_1"] \arrow[r,shift right, "p_2"']
& E
\end{tikzcd}
\end{equation}
are the objects of a category, denoted by $\mathsf{DiscFib}(Eq(p))$, where the morphisms are pairs $(\phi_0, \phi_1)$ of morphisms in $\C$ making the following diagram commute:
\begin{center}
\begin{tikzcd}
R \arrow[rr,shift left] \arrow[rr,shift right]  \arrow[dr,dotted, "\phi_1"] \arrow[dd, "f_1"']
& & F \arrow[dr, dotted,"\phi_0"] \arrow[dd,near end,"f_0"']
& \\
 & R'  \arrow[rr,shift left] \arrow[rr,shift right]  \arrow[dl, "f_1'"']
& & F' \arrow[dl,"f_0'"]\\
Eq(p) \arrow[rr,shift left] \arrow[rr,shift right]
& & E. &
\end{tikzcd}
\end{center}

For a morphism $p \colon E \rightarrow B$ in $\C$, we write $p^{*} \colon \C \downarrow B \rightarrow \C \downarrow E$ for the induced pullback functor along $p$, where $\C \downarrow B$ and $\C \downarrow E$ are the usual slice categories. A morphism $p \colon E \rightarrow B$ is called an \emph{effective descent morphism} when the pullback functor $p^* \colon \C \downarrow B \rightarrow \C \downarrow E$ is \emph{monadic}. Now, this property can also be expressed in terms of discrete fibrations, as follows: $p$ is an effective descent morphism if and only if the functor $K_p \colon \C \downarrow B \rightarrow \mathsf{DiscFib}(Eq(p))$ sending an object $f \colon A \rightarrow B$ in $\C \downarrow B$ to the discrete fibration $(\pi_1,\bar{\pi}_1)$ of equivalence relations
\begin{equation} \label{image-of-K}
\begin{tikzcd}
Eq(\pi_2) \arrow[r,shift left] \arrow[r,shift right] \arrow[d,"\bar{\pi}_1"']
& E \times_B A \arrow[d,"\pi_1"]\\
Eq(p) \arrow[r,shift left] \arrow[r,shift right]
& E,
\end{tikzcd}
\end{equation}
where $\bar{\pi}_1$ is the arrow induced by the universal property of $Eq(p)$ and by the commutativity of \eqref{pullback}, is an equivalence of categories. In a regular category this is equivalent to the following properties \cite{JST}: $p$ is a regular epimorphism and, moreover, for any discrete fibration \eqref{discrete} of equivalence relations with codomain $Eq(p)$, the equivalence relation $R$ is \emph{effective} (i.e. it is a kernel pair).

\subsection*{Factorization systems}

We now recall the link between (reflective) factorization systems and (admissible) Galois structures. For this we mainly follow \cite{Cassidy-Hebert-Kelly, Carboni-Janelidze-Kelly-Pare, Everaert-Gran, Janelidze}, where the reader will find more information about these topics. In this section we shall work in an arbitrary category $\C$.

In order to define the notion of \textit{factorization system}, some notations have to be introduced. For morphisms $e$ and $m$ in $\C$, we write $e \downarrow m$ if there exists, for any commutative square
\begin{center}
\begin{tikzcd}
A \arrow[r,"e"] \arrow[d,"a"'] 
& B \arrow[d,"b"] \arrow[dl,dotted,"\phi"]\\
C \arrow[r,"m"']
& D,
\end{tikzcd}
\end{center}
a unique arrow $\phi : B \rightarrow C$ such that $\phi \cdot e = a$ and $m \cdot \phi = b$. With respect to a given pair $(\E,\M)$ of classes of morphisms in $\C$ one then defines:
\begin{itemize}
\item[$\bullet$] $\E^{\downarrow} = \{ m \in \C \arrowvert e \downarrow m \ \forall e \in \E \}$;
\item[$\bullet$] $\M^{\uparrow} = \{e \in \C \arrowvert e \downarrow m \ \forall m \in \M \}$.
\end{itemize}
\begin{definition}
A \emph{prefactorization system} on the category $\C$ is given by a pair $(\E,\M)$ of classes of morphisms in $\C$ such that $\E = \M^{\uparrow}$ and $\M = \E^{\downarrow}$.
\end{definition}

\begin{definition}
A \emph{factorization system} on $\C$ is a prefactorization system $(\E,\M)$ with the following additional property: for any morphism $f$ in $\C$ there exist morphisms $e \in \E$ and $m \in \M$ such that $f = m\cdot e$.
\end{definition}

Thanks to results from \cite{Cassidy-Hebert-Kelly} we know that, given a full reflective subcategory $\F$ of $\C$
\begin{equation}\label{reflection}
\begin{tikzcd}[column sep = small]
\C \arrow[rr,shift left=0.2cm,"F"] 
& \bot
& \F, \arrow[ll,shift left=0.2cm,"U"]
\end{tikzcd}
\end{equation}
 we then naturally get a prefactorization system $(\E,\M)$ defined as follows:
\begin{itemize}
\item[$\bullet$] $\E = \{f \in \C \, \arrowvert \, F(f) \ \text{is an isomorphism} \}$;
\item[$\bullet$] $\M = \{f \in \C \, \arrowvert \, \text{the following square \eqref{naturality} is a pullback} \}$:
\begin{equation}\label{naturality}
\begin{tikzcd}
A \arrow[r,"\eta_{A}"] \arrow[d,"f"']
& UF(A) \arrow[d,"UF(f)"]\\
B \arrow[r,"\eta_{B}"'] 
& {UF(B),}
\end{tikzcd}
\end{equation}
where $\eta$ is the unit of the adjunction \eqref{reflection}.
\end{itemize}
Moreover, we know that $(\E,\M)$ is a factorization system if the functor $F \colon \C \rightarrow \F$ is \textit{semi-left-exact} in the sense of \cite{Cassidy-Hebert-Kelly}: it preserves all pullbacks of the form
\begin{center}
\begin{tikzcd}
P \arrow[r] \arrow[d]
& U(C) \arrow[d,"U(f)"]\\
B \arrow[r,"\eta_{B}"']
& UF(B),
\end{tikzcd}
\end{center}
where $\eta_{B} : B \rightarrow UF(B)$ is the $B$-component of the unit of the adjunction \eqref{reflection} and $f : C \rightarrow F(B)$ is an arrow in the subcategory $\F$ of $\C$.

In fact a reflection is semi-left-exact if and only if it is \textit{admissible} in the sense of categorical Galois theory \cite{Janelidze} (with respect to the classes of \emph{all} morphisms, as explained in \cite{Carboni-Janelidze-Kelly-Pare}). In this context the morphisms in $\M$ defined above are called \textit{trivial coverings}.

Note that, for a reflector $F \colon \C \rightarrow \F$, there exists a stronger property than being semi-left-exact:

\begin{definition}\cite{Cassidy-Hebert-Kelly} 
A reflector $F : \C \rightarrow \F$ as in \eqref{reflection} has \emph{stable units} when it preserves pullbacks of the form 
\begin{center}
\begin{tikzcd}
P \arrow[r] \arrow[d]
& C \arrow[d,"f"]\\
B \arrow[r,"\eta_{B}"']
& UF(B)
\end{tikzcd}
\end{center}
where $\eta_{B} : B \rightarrow UF(B)$ is the $B$-component of the unit of the adjunction \eqref{reflection} and $f : C \rightarrow UF(B)$ is any arrow in the category $\C$.
\end{definition}
\begin{remark}
\emph{It is well known that, given a torsion theory $(\T,\F)$ in a normal category $\C$, the reflector $F \colon \C \rightarrow \F$ to the torsion-free subcategory has stable units \cite{Everaert-Gran15}. }
\end{remark}

Given the reflection \eqref{reflection} we can define the following two subclasses of morphisms in $\C$:
\begin{itemize}
\item[$\bullet$] $\E' = \{f \in \C \, \arrowvert \, \ \text{the pullback of $f$ along any morphism in $\C$ is in $\E$} \}$;
\item[$\bullet$] $\M^{*} = \{f \in \C \, \arrowvert \, \ \text{there exists an effective descent morphism $p$ such that $p^{*}(f)$ is in $\M$} \}$.
\end{itemize}

Morphisms in $\M^{*}$ are precisely the \textit{coverings} which are the object of study in categorical Galois theory (whenever the reflection \eqref{reflection} is semi-left-exact). In particular, one of the goals of this paper is to describe these coverings in the category of preordered groups, and to show that the pair $(\E', \M^{*})$ is a monotone-light factorization system in the following sense:

\begin{definition}\cite{Carboni-Janelidze-Kelly-Pare} 
A factorization system is said to be \emph{monotone-light} when it is of the form $(\E',\M^{*})$ for some factorization system $(\E,\M)$. 
\end{definition}
We are now ready to state the main result of \cite{Everaert-Gran} (see also \cite{Carboni-Janelidze-Kelly-Pare}), which will be useful later on:
\begin{theorem} \label{thmEG}
Let $\C$ be a normal category. Let $(\T,\F)$ be a torsion theory in $\C$ such that, for any normal monomorphism $k : K \rightarrow A$, the monomorphism $k \cdot \epsilon_{K} : T(K) \rightarrow A$ is normal in $\C$, where $\epsilon_{K} : T(K) \rightarrow K$ is the K-component of the counit $\epsilon$ of the coreflection $\C \rightarrow \T$. We write $(\E,\M)$ for the factorization system associated with the reflector $F : \C \rightarrow \F$, which has stable units. \\
If for any object $C$ in $\C$ there is an effective descent morphism $p : F \rightarrow C$ with $F \in \F$, then $({\E}',{\M}^*)$ is a monotone-light factorization system and, moreover,
\begin{itemize}
\item[$\bullet$] ${\E}'$ is the class of normal epimorphisms in $\C$ whose kernel is in $\T$;
\item[$\bullet$] ${\M}^*$ is the class of morphisms in $\C$ whose kernel is in $\F$. 
\end{itemize} 
\end{theorem}

\subsection*{Limits and short exact sequences in $\mathsf{PreOrdGrp}$}

We recall the description of some limits and colimits, and of the short exact sequences in the category of preordered groups \cite{Clementino-Martins-Ferreira-Montoli}. The product of two preordered groups $(G,P_{G})$ and $(H,P_{H})$ is given by the direct product of groups $G \times H$ with the positive cone $P_{G \times H}$ defined by $P_{G \times H} = P_{G} \times P_{H}$. Next, the equalizer of two arrows $$(f,\bar{f}),(g,\bar{g}) : (G,P_{G}) \rightrightarrows (H,P_{H})$$ is built by computing the equalizer $e : E \rightarrow G$ of $f$ and $g$ in $\mathsf{Grp}$, and the positive cone $P_E$ of $E$ is then given by the intersection of $P_{G}$ and $E$, i.e. the pullback of the inclusion morphisms $P_{G} \rightarrow G$ and $E \rightarrow G$ in the category $\mathsf{Mon}$ of monoids. It is then easily seen that the inclusion
$
\begin{tikzcd}
P_{G} \cap E \arrow[r]
& P_{G}
\end{tikzcd}
$
is the equalizer of $\bar{f}$ and $\bar{g}$ in the category $\mathsf{Mon}$ of monoids. Pullbacks and kernels are computed in the same way, by considering pullbacks and kernels ``componentwise'' at each level (group and monoid). A description of colimits is also possible. The coequalizer of $(f,\bar{f})$ and $(g,\bar{g})$ is computed by taking the coequalizer $q : H \rightarrow Q$ of $f$ and $g$ in $\mathsf{Grp}$ and then the direct image of the submonoid $P_H$ along $q$: $P_{Q} = q(P_{H})$. The description of coproducts is more complicated, and it will not be needed for our work.

The description of normal epimorphisms and normal monomorphisms in the category $\mathsf{PreOrdGrp}$ of preordered groups will also be useful. A morphism $(f,\bar{f}) : (G,P_{G}) \rightarrow (H,P_{H})$ in $\mathsf{PreOrdGrp}$ is an epimorphism if and only if $f$ is surjective. It is a normal epimorphism when, moreover, the morphism $\bar{f}$ in \eqref{morphismPreOrd} is also surjective: $P_{H} = f(P_{G})$. Similarly, $(f,\bar{f})$ is a monomorphism in $\mathsf{PreOrdGrp}$ if and only if $f$ is injective (which also implies that $\bar{f}$ is injective). Such a morphism is a normal monomorphism if $f$ is a normal monomorphism in $\mathsf{Grp}$ and $P_{G} = f^{-1} (P_{H})$ (i.e. the square \eqref{morphismPreOrd} is a pullback).

In the next proposition we gather the information from \cite{Clementino-Martins-Ferreira-Montoli} which is useful to describe short exact sequences in the category $\mathsf{PreOrdGrp}$ of preordered groups:

\begin{proposition} \label{kernel-cokernel}
Consider, in $\mathsf{PreOrdGrp}$, a pair of composable arrows as in the following diagram
\begin{center}
\begin{equation} \label{diag SES}
\begin{tikzcd}[row sep=small,column sep=small]
P_{A} \arrow[rr,"\bar{k}"] \arrow[dd,tail,"a"']
& & P_{B} \arrow[rr,"\bar{f}"] \arrow[dd,tail,"b"]
& & P_{C} \arrow[dd,tail,"c"]\\
 & (P) & & & \\
A \arrow[rr,"k"']
& & B \arrow[rr,"f"']
& & C.
\end{tikzcd}
\end{equation}
\end{center}
Then:
\begin{enumerate}
\item the morphism $(k,\bar{k})$ is the kernel of $(f,\bar{f})$ if and only $k$ is the kernel of $f$ in $\mathsf{Grp}$ and the square (P) is a pullback in $\mathsf{Mon}$;
\item the morphism $(f,\bar{f})$ is the cokernel of $(k,\bar{k})$ if and only if $f$ is the cokernel of $k$ in $\mathsf{Grp}$ and $\bar{f}$ is surjective.
\item the sequence \eqref{diag SES} is a short exact sequence in $\mathsf{PreOrdGrp}$ if and only if 
$$ 
\begin{tikzcd}
0 \arrow[r]
& A \arrow[r,"k"] 
& B \arrow[r,"f"] 
& C  \arrow[r]
& 0
\end{tikzcd}
$$
 is a short exact sequence in $\mathsf{Grp}$, (P) is a pullback in $\mathsf{Mon}$, and $\bar{f}$ is surjective.
\end{enumerate} 
\end{proposition}

The category $\mathsf{PreOrdGrp}$ of preordered groups is \emph{normal}, as observed in \cite{Clementino-Martins-Ferreira-Montoli}, where it is also proved that a morphism in $\mathsf{PreOrdGrp}$ is \emph{effective for descent} if and only if it is a normal epimorphism (or, equivalently, if and only if it is a regular epimorphism).

\subsection*{Schreier points and special Schreier morphisms in monoids}

As we saw in the introduction the category of preordered groups is equivalent to the one whose objects are pairs $(G,M)$ where $G$ is a group and $M$ is a submonoid of $G$ closed under conjugation. While the category of groups is protomodular, which means that the Split Short Five Lemma holds in $\mathsf{Grp}$, this is not the case for the category of monoids \cite{BB}. Moreover, actions in monoids are not equivalent to split extensions of monoids (while this is the case for groups). Nevertheless, we can restrict our attention to a class of points (a ``point'' $(A,B,p,s)$ being a split epimorphism $p \colon A \rightarrow B$ with fixed section $s \colon B \rightarrow A$) in $\mathsf{Mon}$, called \textit{Schreier points} (or, equivalently, \textit{Schreier split epimorphisms}) \cite{BMFMS}, which have a behavior which is quite similar to the ones in the category of groups. The class of Schreier points corresponds to monoid actions, and it was shown in \cite{BMFMS} that the Split Short Five Lemma does hold for such points. 

\begin{definition}
A \emph{Schreier point} in the category $\mathsf{Mon}$ of monoids is a point $(A,B,p,s)$ such that for any element $a$ in $A$ there exists a unique element $x$ in the kernel $\mathsf{Ker}(p)$ of $p$ such that
\[a = x + (s \cdot p) (a).\]
\end{definition} 

This kind of points are useful to ``locally'' extend some classical properties of split extensions of groups to the context of monoids. We shall not develop these interesting aspects here, but we refer the reader to \cite{BMFMS} for a thorough introduction to this subject. What will be of interest for the purpose of this paper is to briefly recall the properties of \textit{special Schreier morphisms} in the category of monoids.

\begin{definition} \hspace*{3cm}
\begin{itemize}
\item An internal reflexive relation in the category $\mathsf{Mon}$ of monoids
\begin{center}
\begin{tikzcd}
R \arrow[r,shift left=2.3,"r_{1}"] \arrow[r,shift right=2.3,"r_{2}"']
& A \arrow[l,"s"' description]
\end{tikzcd}
\end{center}
is said to be a \emph{Schreier reflexive relation} when the point $(R,A,r_{1},s)$ is a Schreier one.
\item A morphism $f : A \rightarrow B$ in the category $\mathsf{Mon}$ of monoids is said to be a \emph{special Schreier morphism} when its kernel pair
\begin{center}
\begin{tikzcd}[column sep = large]
Eq(f) \arrow[r,shift left=3.5,"f_{1}"] \arrow[r,shift right=3.5,"f_{2}"']
& A \arrow[l,"{(1,1)}"' description]
\end{tikzcd}
\end{center}
is a Schreier reflexive relation, where $(1,1) \colon A \rightarrow Eq(f)$ is such that $f_1\cdot (1,1)= 1=f_2  \cdot (1,1)$.
\end{itemize}
\end{definition}

It is then possible to prove \cite{BMFMS} that any surjective special Schreier morphism $f \colon A \rightarrow B$ is the cokernel of its kernel. Accordingly, we get an extension of monoids:
\begin{center}
\begin{tikzcd}
0 \arrow[r]
& \mathsf{Ker}(f) \arrow[r,"k"]
& A \arrow[r,"f"]
& B \arrow[r]
&0.
\end{tikzcd}
\end{center}
These \textit{special Schreier extensions} satisfy some remarkable properties. The following proposition states two of them, which will be needed for our future investigations:

\begin{proposition} \label{proposition S extensions} \cite{BMFMS} \hspace*{3cm}
\begin{enumerate}
\item Special Schreier extensions are pullback stable in $\mathsf{Mon}$.
\item The \emph{Short Five Lemma} holds for special Schreier extensions. This means that given any commutative diagram \eqref{generic-ses} of short exact sequences in $\mathsf{Mon}$, where $f$ and $f'$ are special Schreier morphisms, and $a$ and $c$ are isomorphisms, then $b$ is also an isomorphism.
\end{enumerate}
\end{proposition}

\section{Coverings in the category of preordered groups}

\begin{proposition} \label{torsion OrdGrp}
The pair of full (replete) subcategories \emph{($\mathsf{Grp},\mathsf{ParOrdGrp}$)} of $\mathsf{PreOrdGrp}$ is a torsion theory in the normal category $\mathsf{PreOrdGrp}$.
\end{proposition}

\begin{proof}
Let us first show that the only arrow in $\mathsf{PreOrdGrp}$ from an object of $\mathsf{Grp}$ to an object of $\mathsf{ParOrdGrp}$ is the zero morphism. Consider an arrow $(f,\bar{f}) : (G,G) \rightarrow (H,P_{H})$ in $\mathsf{PreOrdGrp}$, with $(G,G)$ in $\mathsf{Grp}$ and $(H,P_{H})$ in $\mathsf{ParOrdGrp}$:
\begin{center}
\begin{tikzcd}
G \arrow[r,"\bar{f}"] \arrow[d,equal]
& P_{H} \arrow[d]\\
G \arrow[r,"f"']
& H.
\end{tikzcd}
\end{center}
For any $x \in G$, its opposite $-x$ is also in $G$,  and
\[0 = \bar{f}(x-x) = \bar{f}(x) + \bar{f}(-x),\]
with $\bar{f}(x),\bar{f}(-x) \in P_{H}$, and $P_{H}$ is a reduced monoid. This implies that $\bar{f}(x) = \bar{f}(-x) = 0$, $\bar{f} = 0$, and then $f = 0$. 

Consider then an object $(G,P_{G})$ of $\mathsf{PreOrdGrp}$, and define
\[N_{G} = \{n \in G \ \arrowvert \ n \in P_{G} \ \text{and} \ -n \in P_{G}\}.\]
It is a normal subgroup of $G$: indeed, if $n \in N_{G}$ and $x \in G$, then we have that $x+n-x \in P_{G}$ and $-(x+n-x) = x-n-x \in P_{G}$, since the submonoid $P_{G}$ is closed under conjugation in $G$. Accordingly, the sequence
\begin{tikzcd}
N_{G} \arrow[r,tail,"k_{G}"] 
& G \arrow[r,two heads,"\eta_{G}"]
& G/N_{G}
\end{tikzcd}
is a short exact sequence in the category $\mathsf{Grp}$ of groups. Consider next the direct image factorization of the morphism $\eta_{G}\cdot g$ in the category $\mathsf{Mon}$ of monoids, where 
\begin{tikzcd}
P_{G} \arrow[r,tail,"g"]
& G
\end{tikzcd}
is the inclusion: 
\begin{center}
\begin{tikzcd}
P_{G} \arrow[r,two heads,"\bar{\eta}_{G}"] \arrow[d,tail,"g"']
& \eta_{G}(P_{G}) \arrow[d,tail,"\psi_{G}"]\\
G \arrow[r,two heads,"\eta_{G}"']
& G/N_{G}.
\end{tikzcd}
\end{center}
Let us now prove that the sequence
\begin{center}
\begin{equation} \label{diagram sequence}
\begin{tikzcd}
N_{G} \arrow[r,tail,"\bar{k}_{G}"] \arrow[d,equal]
& P_{G} \arrow[r,two heads,"\bar{\eta}_{G}"] \arrow[d,tail,"g"']
& \eta_{G}(P_{G}) \arrow[d,tail,"\psi_{G}"]\\
N_{G} \arrow[r,tail,"k_{G}"']
& G \arrow[r,two heads,"\eta_{G}"']
& G/N_{G}
\end{tikzcd}
\end{equation}
\end{center}
is exact in the category $\mathsf{PreOrdGrp}$ of preordered groups. This follows from Proposition \eqref{kernel-cokernel}, since the left-hand square in \eqref{diagram sequence} is clearly a pullback in $\mathsf{Mon}$, the lower sequence is exact, and the morphism $\bar{\eta}_{G}$ is surjective by construction.

It is obvious that $(N_{G},N_{G}) \in \mathsf{Grp}$, so that the proof will be complete if we show that $(G/N_{G},\eta_{G}(P_{G}))$ is in $\mathsf{ParOrdGrp}$. Now, if $y + z = 0$ for $y,z \in \eta_{G}(P_{G})$, then there exist $x,x' \in P_{G}$ such that $\eta_{G}(x) = y$ and $\eta_{G}(x') = z$, so that $\eta_{G}(x + x') = y + z = 0$, that is $x + x' \in N_{G}$. Since $N_{G}$ is a group and $P_{G}$ is a monoid it follows that
\[-x = x' - x' - x = x' - (x + x') \in P_{G},\]
hence $x \in N_{G}$, which implies that $y = \eta_{G}(x) = 0$ and therefore $z = 0$. Accordingly, the submonoid $\eta_{G}(P_{G})$ is reduced, and $(G/N_{G},\eta_{G}(P_{G}))$ a partially ordered group.  
\end{proof}

\begin{remark}
\emph{Note that a similar result can be proved in the category of commutative monoids, as observed in \cite{Facchini} (Section 2.3): the pair $(\mathsf{Ab},\mathsf{RedCMon})$ is a torsion theory in the category $\mathsf{CMon}$ of commutative monoids, where we write $\mathsf{Ab}$ for the category of abelian groups and $\mathsf{RedCMon}$ for the category of reduced commutative monoids.}
\end{remark}

As a consequence of Proposition \ref{torsion OrdGrp} we get the following result:

\begin{corollary} \label{corollary adjunction} \hspace*{3cm}
\begin{itemize}
\item The category $\mathsf{ParOrdGrp}$ is reflective in the category $\mathsf{PreOrdGrp}$
\begin{equation}\label{main-adj}
\begin{tikzcd}[column sep = tiny]
\mathsf{PreOrdGrp} \arrow[rr,shift left=0.2cm,"F"]
& \bot
& \mathsf{ParOrdGrp}, \arrow[ll,shift left=0.2cm,hook',"U"]
\end{tikzcd}
\end{equation}
and each component of the unit $\eta$ of the adjunction (as in \eqref{diagram sequence}) is a normal epimorphism.
\item The category $\mathsf{Grp}$ is coreflective in $\mathsf{PreOrdGrp} $ and each component of the counit $\kappa$ of the adjunction (as in \eqref{diagram sequence}) is a normal monomorphism.
\end{itemize}
\end{corollary}

\begin{proof}
This follows from the Proposition \ref{torsion OrdGrp} and the (only) Proposition in \cite{JT} (see also \cite{Bourn-Gran}, and \cite{CDT}).
\end{proof}

We now make some useful comments on the short exact sequence \eqref{diagram sequence} constructed in the proof of Proposition \ref{torsion OrdGrp}:

\begin{lemma} \label{lemma M}
Consider the following commutative diagram in the category $\mathsf{Mon}$ of monoids, where the two rows are special Schreier extensions:
\begin{equation} \label{diagram iso pb}
\begin{tikzcd}
0 \arrow[r]
& \mathsf{Ker}(f) \arrow[r,"k"] \arrow[d,"a"']
& A \arrow[r,"f"] \arrow[d,"b"']
& B \arrow[r] \arrow[d,"c"]
& 0\\
0 \arrow[r]
&  \mathsf{Ker}(f') \arrow[r,"k'"'] 
& A' \arrow[r,"f'"']
& B' \arrow[r]
&0.
\end{tikzcd}
\end{equation}
If the morphism $a$ is an isomorphism, then the right-hand square of diagram \eqref{diagram iso pb} is a pullback.
\end{lemma}

\begin{proof}
Let us consider the pullback $(P,p_{A'},p_{B})$ of $f'$ and $c$, and let $(\mathsf{Ker}(p_{B}),k'')$ be the kernel of $p_{B}$.
\begin{center}
\begin{tikzcd}
0 \arrow[rr] 
& 
&  \mathsf{Ker}(f) \arrow[rr,"k"] \arrow[dd,"a"'] \arrow[dr,dotted,"\gamma"]
& 
& A \arrow[rr,"f"] \arrow[dd,near start,"b"'] \arrow[dr,dotted,"\phi"]
&
& B \arrow[rr] \arrow[dd,near start,"c"] \arrow[dr,equal]
&
& 0\\
 & & &  \mathsf{Ker}(p_{B}) \arrow[rr,near end,"k''"] \arrow[dl,dotted,"\psi"]
& & P \arrow[rr,near end,"p_{B}"] \arrow[dl,"p_{A'}"]
& & B \arrow[dl,"c"]
& \\
0 \arrow[rr]
&
&  \mathsf{Ker}(f')\arrow[rr,"k'"']
&
& A' \arrow[rr,"f'"']
&
& B' \arrow[rr]
&
& 0
\end{tikzcd}
\end{center} 
We are going to show that the arrow $\phi$ induced by the universal property of the pullback $(P,p_{A'},p_{B})$ is an isomorphism. By the universal property of the kernel $k' = \ker(f')$ we first get the morphism $\psi$ such that $k' \cdot \psi = p_{A'}\cdot k''$. In the same way the universal property of the kernel $k'' = \ker(p_{B})$ gives a unique arrow $\gamma$ such that $k'' \cdot \gamma = \phi \cdot k$. It is easily seen that $\psi \cdot \gamma = a$, with $a$ an isomorphism by assumption. In addition, $\psi$ is also an isomorphism since $(P,p_{A'},p_{B})$ is a pullback, so that $\gamma$ is itself an isomorphism. We have that the bottom row is a special Schreier extension, hence by the pullback stability of special Schreier extensions (Proposition \ref{proposition S extensions}) we get that the middle row of the above diagram is a special Schreier extension. Again by Proposition \ref{proposition S extensions} (second assertion) we  apply the Short Five Lemma to the diagram
\begin{center}
\begin{tikzcd}
0 \arrow[r]
&  \mathsf{Ker}(f) \arrow[r,"k"] \arrow[d,"\gamma"']
& A \arrow[r,"f"] \arrow[d,"\phi"']
& B \arrow[r] \arrow[d,equal]
& 0\\
0 \arrow[r]
&  \mathsf{Ker}(p_{B}) \arrow[r,"k''"'] 
& P \arrow[r,"p_{B}"']
& B \arrow[r]
&0,
\end{tikzcd}
\end{center}
to conclude that $\phi$ is an isomorphism, that is the right-hand square in the diagram \eqref{diagram iso pb} is a pullback.
\end{proof}

\begin{corollary} 
Consider a short exact sequence in $\mathsf{PreOrdGrp}$ of the following form:
\begin{center}
\begin{equation} \label{diag SES pb}
\begin{tikzcd}
0 \arrow[r]
& K \arrow[r,"\bar{k}"] \arrow[d,equal]
& P_{G} \arrow[r,"\bar{f}"] \arrow[d,tail]
& P_{H} \arrow[d,tail] \arrow[r] & 0 \\
 0 \arrow[r]
& K \arrow[r,"k"']
& G \arrow[r,"f"']
& H \arrow[r] & 0
\end{tikzcd}
\end{equation}
\end{center}
Then the upper sequence is a special Schreier extension, and the right-hand square in \eqref{diag SES pb} is a pullback in the category $\mathsf{Mon}$ of monoids.
\end{corollary}

\begin{proof}
By Lemma \ref{lemma M}, since any short exact sequence in the category $\mathsf{Grp}$ of groups is a special Schreier extension, it suffices to show that the sequence
\begin{tikzcd}
K \arrow[r,"\bar{k}"] 
& P_{G} \arrow[r,"\bar{f}"]
& P_{H}
\end{tikzcd}
is a special Schreier extension in $\mathsf{Mon}$. Consider the kernel pair $(Eq(\bar{f}),r_{1},r_{2})$ of $\bar{f}$ where the projections $r_{1}$ and $r_{2}$ are split by the diagonal morphism $s \colon P_G \rightarrow Eq(\overline{f})$. Note that there is no restriction in assuming that $f \colon G \rightarrow H \cong G/K$ is the canonical quotient of $G$ by its normal subgroup $K$, and we write $f(g)= \bar{g}^K$.
For any $(a,b) \in Eq(\bar{f})$, one has the equalities $f(a) = \bar{f}(a) = \bar{f}(b) = f(b),$ since $\bar{f}$ is the restriction of $f$ to the positive cone $P_{G}$ of $G$. It follows that $\bar{a}^{K} = \bar{b}^{K}$, and there exists $x \in K$ such that $b = x + a$. As a consequence $(0,x) \in \mathsf{Ker}(r_1)$ satisfies the equalities
\[(0,x) + (s \cdot r_{1})(a,b) = (0,x) + (a,a) = (a,x+a) = (a,b),\]
and $(0,x)$ is the only element of $\mathsf{Ker}(r_1)$ with this property. This shows that $(r_{1},s)$ is a Schreier point, and the arrow $\bar{f} : P_{G} \rightarrow P_{H}$ is a special Schreier morphism. Since any surjective special Schreier morphism is the cokernel of its kernel and since $\bar{f}$ is surjective, this shows that $K \rightarrow P_{G} \rightarrow P_{H}$ is a special Schreier extension.
\end{proof}

In particular the previous result implies the following
\begin{corollary} \label{lemma 2 M}
Consider the short exact sequence \eqref{diagram sequence} in $\mathsf{PreOrdGrp}$.
Then the square 
$$\begin{tikzcd}
 P_{G} \arrow[r,two heads,"\bar{\eta}_{G}"] \arrow[d,tail,"g"']
& \eta_{G}(P_{G}) \arrow[d,tail,"\psi_{G}"]\\
G \arrow[r,two heads,"\eta_{G}"']
& G/N_{G}
\end{tikzcd}
$$
is a pullback in the category $\mathsf{Mon}$ of monoids.
\end{corollary}

As a consequence of Corollary \ref{lemma 2 M} from now on we will write $P_{G}/N_{G}$ instead of $\eta_{G}(P_{G})$ for the codomain of $\bar{\eta}_G$ in the short exact sequence \eqref{diagram sequence}. This means that in this sequence in $\mathsf{PreOrdGrp}$ we have short exact sequences both at the group and at the monoid level.

As reminded in the previous section, if the reflector $F \colon \mathsf{PreOrdGrp} \rightarrow \mathsf{ParOrdGrp}$ has stable units, then the adjunction \eqref{main-adj} gives rise to a factorization system and is admissible in the sense of the categorical Galois theory \cite{Janelidze}. The following proposition states that this is in fact the case for our adjunction \eqref{main-adj}.

\begin{proposition}\label{Stable-units}
The reflector $F  \colon \mathsf{PreOrdGrp} \rightarrow \mathsf{ParOrdGrp}$ in the adjunction \eqref{main-adj} has stable units.
\end{proposition}

\begin{proof}
We have to prove that the functor $F$ preserves pullbacks of the form of the right-hand cube in the following commutative diagram
\begin{center}
\begin{equation} \label{diagram pb}
\begin{tikzcd}
& & P_{G} \times_{P_{G}/N_{G}} P_{H} \arrow[rr,"\bar{p}_{2}"] \arrow[dr,tail,"\phi"] \arrow[dd,"\bar{p}_{1}"']
& & P_{H} \arrow[dd,near start,"\bar{f}"] \arrow[dr,tail,"h"]
& \\
& & & G \times_{G/N_{G}} H \arrow[dd,near end,"p_{1}"] \arrow[rr,near end,"p_{2}"']
& & H \arrow[dd,"f"]\\
N_{G} \arrow[rr,tail,dotted,"\bar{k}_{G}"] \arrow[uurr,tail,dotted,near start,"\bar{i}"] \arrow[dr,equal]
& & P_{G} \arrow[rr,two heads,"\bar{\eta}_{G}"] \arrow[dr,tail,"g"']
& & P_{G}/N_{G} \arrow[dr,tail,"\psi_{G}"]
& \\
 & N_{G} \arrow[uurr,tail,dotted,near start,"i"] \arrow[rr,tail,dotted,"k_{G}"']
& & G \arrow[rr,two heads,"\eta_{G}"']
& & G/N_{G}
\end{tikzcd}
\end{equation}
\end{center}
in which $(k_{G},\bar{k}_{G})$ is the kernel of $(\eta_{G},\bar{\eta}_{G})$, and the induced arrow $(i,\bar{i})$ is the kernel of the arrow $(p_{2},\bar{p}_{2})$.
If we apply the functor $F$ to the diagram \eqref{diagram pb} we get the commutative diagram
\begin{center}
\begin{equation} \label{diagram pb2}
\begin{tikzcd} [column sep =scriptsize,row sep =scriptsize]
 & & \left( P_{G} \times_{P_{G}/N_{G}} P_{H} \right)/N \arrow[rr,"F(\bar{p}_{2})"] \arrow[dr,tail,"\psi"] \arrow[dd,"F(\bar{p}_{1})"']
& & P_{H}/N_{H} \arrow[dr,tail,"\psi_{H}"] \arrow[dd,near start,"F(\bar{f})"]
& \\
& & & \left( G \times_{G/N_{G}} H \right)/N \arrow[rr,near end,"F(p_{2})"'] \arrow[dd,near end,"F(p_{1})"]
& & H/N_{H} \arrow[dd,"F(f)"]\\
0 \arrow[uurr,tail,dotted,near start,"F(\bar{i})"] \arrow[rr,tail,dotted] \arrow[dr,equal]
& & P_{G}/N_{G} \arrow[rr,dash,shift left=0.03cm] \arrow[rr,dash,shift right=0.03cm] \arrow[dr,tail,"\psi_{G}"']
& & P_{G}/N_{G} \arrow[dr,tail,"\psi_{G}"]
& \\
 & 0 \arrow[uurr,tail,dotted,near start,"F(i)"] \arrow[rr,tail,dotted]
& & G/N_{G} \arrow[rr,equal]
& & G/N_{G}
\end{tikzcd}
\end{equation}
\end{center}
where we write $N$ for $N_{G \times_{G/N_{G}} H}$ and $(F(a),F(\bar{a}))$ for the image by $F$ of any arrow $(a,\bar{a})$ in $\mathsf{PreOrdGrp}$. We observe that the arrows $p_{2}$ and $\bar{p}_{2}$ are normal epimorphisms since $\eta_{G}$ and $\bar{\eta}_{G}$ are normal epimorphisms and the front and back squares of the cube in \eqref{diagram pb} are pullbacks in $\mathsf{Grp}$ and $\mathsf{Mon}$, respectively. This means that $(p_{2},\bar{p}_{2})$ is the cokernel of its kernel: $(p_{2},\bar{p}_{2}) = \mathsf{coker}(i,\bar{i})$. It follows that $(F(p_{2}),F(\bar{p}_{2}))$ is the cokernel of $(F(i),F(\bar{i}))$, which is the zero arrow, and the arrow $(F(p_{2}),F(\bar{p}_{2}))$ is then an isomorphism. Accordingly, the front and the back squares of the cube in the diagram \eqref{diagram pb2} are pullbacks in $\mathsf{Grp}$ and in $\mathsf{Mon}$, respectively, i.e. the cube of this diagram is a pullback in $\mathsf{ParOrdGrp}$, as desired.
\end{proof}

\begin{remark}
\emph{By taking into account the fact that the pair ($\mathsf{Grp},\mathsf{ParOrdGrp}$) is a torsion theory in the normal category $\mathsf{PreOrdGrp}$ (thanks to Proposition \ref{torsion OrdGrp}) the above result can be deduced from Theorem 1.6 in \cite{Everaert-Gran15}. We have included a direct proof here in order to make the article more self-contained, and also to give an explicit description of the behavior of the reflector $F  \colon \mathsf{PreOrdGrp} \rightarrow \mathsf{ParOrdGrp}$.}
\end{remark}

Let us now characterize the two classes $\E$ and $\M$ of the factorization system induced by the reflector $F  \colon \mathsf{PreOrdGrp} \rightarrow \mathsf{ParOrdGrp}$ in the category $\mathsf{PreOrdGrp}$ of preordered groups:

\begin{proposition}
Given the adjunction \eqref{main-adj}, we have a factorization system $(\E,\M)$ in $\mathsf{PreOrdGrp}$ where:
\begin{itemize}
\item[$\bullet$] $(f,\tilde{f}) : (G,P_{G}) \rightarrow (H,P_{H})$ is in the class $\E$ if and only if the following conditions hold:
\begin{enumerate}
\item[(a)] $f^{-1}(N_{H}) = N_{G}$,
\item[(b)] for any $y \in H$ there exists $x \in G$ such that $\overline{f(x)}^{N_{H}} = \overline{y}^{N_{H}}$,
\item[(c)]  for any $y \in P_{H}$ there exists $x \in P_{G} $ such that $\overline{\tilde{f}(x)}^{N_{H}} = \overline{y}^{N_{H}}$. 
\end{enumerate}
\item[$\bullet$] $(f,\tilde{f}) : (G,P_{G}) \rightarrow (H,P_{H})$ is in the class $\M$ if and only if the morphism $\phi : N_{G} \rightarrow N_{H}$ (which is the restriction of $f \colon G \rightarrow H$ to $N_{G}$) is an isomorphism.
\end{itemize}
\end{proposition}

\begin{proof}
Consider the commutative diagram
\begin{center}
\begin{equation} \label{diagram E and M}
\begin{tikzcd}
 & & & P_{H} \arrow[rr,"\bar{\eta}_{H}"] \arrow[dd]
& & P_{H}/N_{H} \arrow[dd]\\
 & & P_{G} \arrow[ur,"\tilde{f}"'] \arrow[rr,near end,"\bar{\eta}_{G}"] \arrow[dd]
& & P_{G}/N_{G} \arrow[ur,"\tilde{\alpha}"'] \arrow[dd]
& \\
 & N_{H} \arrow[uurr,bend left,"\ker(\bar{\eta}_{H})"] \arrow[rr,near start,"\ker(\eta_{H})"]
& & H \arrow[rr,near end,"\eta_{H}"] 
& & H/N_{H} \\
N_{G} \arrow[uurr,bend left,"\ker(\bar{\eta}_{G})"] \arrow[ur,"\phi"] \arrow[rr,"\ker(\eta_{G})"']
& & G \arrow[ur,"f"] \arrow[rr,"\eta_{G}"']
& & G/N_{G} \arrow[ur,"\alpha"']
& 
\end{tikzcd}
\end{equation}
\end{center}
where $(\alpha,\tilde{\alpha})$ stands for $F(f,\tilde{f})$, and where the front and the back squares of the cube are the $(G,P_{G})$-component and the $(H,P_{H})$-component of the unit of the adjunction \eqref{main-adj}, respectively.
\begin{itemize}
\item[$\bullet$] Assume that $(f,\tilde{f}) : (G,P_{G}) \rightarrow (H,P_{H})$ is in the class $\E$, so that $(\alpha,\tilde{\alpha})$ is an isomorphism in $\mathsf{PreOrdGrp}$. The fact that $\alpha$ is a monomorphism implies that the square
\begin{center}
\begin{equation} \label{pb square}
\begin{tikzcd}
N_{G} \arrow[r,"\ker(\eta_{G})"] \arrow[d,"\phi"']
& G \arrow[d,"f"]\\
N_{H} \arrow[r,"\ker(\eta_{H})"']
& H
\end{tikzcd}
\end{equation}
\end{center}
is a pullback, i.e. $f^{-1}(N_{H}) = N_{G}$ (by Lemma \ref{property-normal}(2)). Furthermore, knowing that $\alpha$ is surjective, for any $y$ in $H$ there exists an $x$ in $G$ such that $\alpha(\bar{x}^{N_{G}}) = \bar{y}^{N_{H}}$, that is $\overline{f(x)}^{N_{H}} = \overline{y}^{N_{H}}$. We can show the analogue assertion for $\tilde{f}$ in a similar way, since $\tilde{\alpha}$ is surjective.\\
Conversely, if $f^{-1}(N_{H}) = N_{G}$, the square \eqref{pb square} is a pullback, and this implies that $\alpha$ is a monomorphism (by Lemma \ref{property-normal}(2) in the category $\mathsf{Grp}$). Now, the assumption $(b)$ guarantees that $\eta_H \cdot f = \alpha \cdot \eta_G$ is surjective, hence $\alpha$ is surjective. Similarly, $\tilde{\alpha}$ is surjective because the assumption $(c)$ says that $\bar{\eta}_H \cdot \tilde{f} = \tilde{\alpha} \cdot \bar{\eta}_{G}$ is surjective.
It follows that $(\alpha,\tilde{\alpha})$ is an isomorphism in $\mathsf{PreOrdGrp}$, i.e. that $(f,\tilde{f})$ belongs to the class $\E$.
\item[$\bullet$] To prove the second point we observe that the arrow $(f,\tilde{f})$ belongs to the class $\M$ if and only if the cube in the diagram \eqref{diagram E and M} is a pullback in $\mathsf{PreOrdGrp}$, and this is equivalent to the bottom and the top squares of this cube being pullbacks in the categories $\mathsf{Grp}$ and $\mathsf{Mon}$, respectively.\\
If these two squares are pullbacks then the induced arrow $\phi : N_{G} \rightarrow N_{H}$, the restriction of the morphism $f : G \rightarrow H$ to $N_{G}$, is obviously an isomorphism.\\
Conversely, if $\phi$ is an isomorphism, then the bottom square of the cube is a pullback (since the Short Five Lemma holds in the category $\mathsf{Grp}$ of groups). The fact that the top square of the same cube is a pullback (in the category $\mathsf{Mon}$ of monoids) is a consequence of Corollary \ref{lemma 2 M}. Indeed, this latter states that the front and the back squares in the cube of diagram \eqref{diagram E and M} are pullbacks, hence the top square is a pullback since the bottom one is a pullback.
 \qedhere 
\end{itemize} 
\end{proof}

Now that we have a description of the \emph{trivial coverings}, i.e. the morphisms in the class $\M$, we would like to have a description of the class ${\M}^*$ of \emph{coverings}.
We shall actually prove that there is a monotone-light factorization system $(\E',\M^{*})$, by applying Theorem \ref{thmEG}. In the following two propositions
 we verify the two fundamental assumptions needed to apply that theorem:
 
\begin{proposition} \label{ParOrdGrp covers PreOrdGrp}
For any object $(G,P_{G})$ in the category $\mathsf{PreOrdGrp}$ of preordered groups, there exist an object $(H,P_{H})$ in the subcategory $\mathsf{ParOrdGrp}$ of partially ordered groups and an effective descent morphism $$(f,\bar{f}) : (H,P_{H}) \rightarrow (G,P_{G})$$ from $(H,P_{H})$ to $(G,P_{G})$.
\end{proposition}

\begin{proof}
Let $(G,P_{G}) \in \mathsf{PreOrdGrp}$. Define $(H,P_{H})$ in the following way:
\begin{itemize}
\item[$\bullet$] $H = \Z \times G$;
\item[$\bullet$] $P_{H} = \left( \N \times P_{G} \right) \backslash \{(0,g) \ \arrowvert \ g \neq 0\}$.
\end{itemize}
It is easy to check that $P_{H}$ is a submonoid of the group $H$ (endowed with the natural group structure). To show that $P_{H}$ is closed in $H$ under conjugation consider $(z,g) \in H$ and $(n,h) \in P_{H}$. Then \[(z,g) + (n,h) - (z,g) = (z + n - z, g + h - g) \in \N \times P_{G}\]
since $P_{G}$ is closed under conjugation in $G$, and if $z + n - z = 0$, i.e. if $n = 0$, then $(n,h) \in P_{H}$ implies $h = 0$, that is $g + h - g = g-g = 0$. This means that $(z,g) + (n,h) - (z,g) \in P_{H}$, and $(H,P_{H})$ is a preordered group, which actually lies in $\mathsf{ParOrdGrp}$, since by definition the submonoid $P_{H}$ is reduced (the only element having an inverse in $P_{H}$ is $(0,0)$).

Let us next consider the function $f : H \rightarrow G$ defined, for any $(z,g) \in H$, by $f(z,g) = g$.
It is a morphism in $\mathsf{PreOrdGrp}$, since it is a group morphism and $f(z,g) = g \in P_{G}$, for any $(z,g) \in P_{H}$. In other words, the restriction $\bar{f}$ of $f$ to $P_{H}$ takes its values in $P_{G}$. The morphism $(f,\bar{f}) : (H,P_{H}) \rightarrow (G,P_{G})$ is also a normal epimorphism in $\mathsf{PreOrdGrp}$, since both $f$ and $\bar{f}$ are easily seen to be surjective. Since effective descent morphisms coincide with normal epimorphisms in $\mathsf{PreOrdGrp}$ \cite{Clementino-Martins-Ferreira-Montoli}, the proof is complete.
\end{proof}

\begin{proposition}
For any normal monomorphism $(i,\bar{i}) : (K,P_{K}) \rightarrow (G,P_{G})$ in $\mathsf{PreOrdGrp}$, the monomorphism $(i,\bar{i})\cdot (k,\bar{k}) : (N_{K},N_{K}) \rightarrow (G,P_{G})$ is normal, where $(k,\bar{k}) : (N_{K},N_{K}) \rightarrow (K,P_{K})$ is the $(K,P_{K})$-component of the counit of the coreflection $T : \mathsf{PreOrdGrp} \rightarrow \mathsf{Grp}$.
\end{proposition}

\begin{proof}
Since $(i,\bar{i})$ is a normal monomorphism, there exists an arrow $(f,\bar{f}) : (G,P_{G}) \rightarrow (H,P_{H})$ in $\mathsf{PreOrdGrp}$ such that $(i,\bar{i}) = \ker(f,\bar{f})$. With that notation we then have that $K = \mathsf{Ker}(f)$ and that $P_{K} = K \cap P_{G}$ since kernels in $\mathsf{PreOrdGrp}$ are computed componentwise at the level of groups and monoids, respectively.

Let us first show that $N_{K}$ is normal in $G$, where
\begin{eqnarray} N_{K} & = &  \{x \in K \arrowvert x \in P_{K} \ \text{and} \ -x \in P_{K}\} \nonumber \\ &=& \{x \in K \,  \arrowvert \, x \in K \cap P_{G} \ \text{and} \ -x \in K \cap P_{G}\} \nonumber \\ &=& K \cap N_G.\nonumber
\end{eqnarray}
Since $K$ and $N_G$ are two normal subgroups of $G$, $N_K$ is normal in $G$.
In order to prove that the inclusion $(i,\bar{i})\cdot (k,\bar{k}) : (N_{K},N_{K}) \rightarrow (G,P_{G})$ is a normal monomorphism in $\mathsf{PreOrdGrp}$ one observes that the following rectangle is a pullback
$$
\begin{tikzcd}
N_{K} \arrow[r,tail,"\bar{k}"] \arrow[d,equal]
& P_{K} \arrow[r,tail,"\bar{i}"] \arrow[d,tail,""']
& P_{G} \arrow[d,tail,"\psi_{G}"]\\
N_{K} \arrow[r,tail,"k"']
& K \arrow[r,tail,"i"']
& G, 
\end{tikzcd} 
$$
since it is made of two pullbacks, and the result then follows from Proposition \ref{kernel-cokernel}.
\end{proof}

We are now ready to state the final result of this section.

\begin{theorem}\label{description-coverings}
Let us consider the following classes of morphisms in $\mathsf{PreOrdGrp}$:
\begin{itemize}
\item[$\bullet$] ${\E}'  = \{(f,\bar{f}) \in \mathsf{PreOrdGrp} \, \arrowvert \, (f,\bar{f}) \ \text{is a normal epimorphism such that } \mathsf{Ker}(f,\bar{f}) \in \mathsf{Grp}\}$;
\item[$\bullet$] ${\M}^* = \{(f,\bar{f}) \in \mathsf{PreOrdGrp} \, \arrowvert \, \mathsf{Ker}(f,\bar{f}) \in \mathsf{ParOrdGrp}\}$.
\end{itemize}
Then $({\E}' ,{\M}^*)$ is a monotone-light factorization system.
\end{theorem}

\begin{proof}
This result follows from Theorem \ref{thmEG}, which can be applied to the 
reflection \ref{main-adj}  thanks to
the two previous propositions.
\end{proof}
The \emph{coverings} with respect to the adjunction \eqref{main-adj} are then the morphisms $f \colon A \rightarrow B$ in $\mathsf{PreOrdGrp}$ such that $\mathsf{Ker}(f) \in \mathsf{ParOrdGrp}$. This description is then similar to the one of the \emph{locally semisimple coverings} relative to a generalized semisimple class, given by Janelidze, M\'{a}rki and Tholen in \cite{JMT}. We explain the link with this latter approach in the next section.

\section{Classification of the coverings of preordered groups}

Let us first recall the approach to locally semisimple coverings based on Galois theory developed in \cite{JMT}. Here below we shall adapt the context in order to include the example of the category $\mathsf{PreOrdGrp}$ of preordered groups.

Let $\C$ be any normal category in which normal epimorphisms and effective descent morphisms coincide. Let us consider a fixed class $\mathcal{X}$ of objects in $\C$, called a \textit{generalized semisimple class}, having the property that the following two properties hold for any pullback
 \begin{center}
\begin{tikzcd}
E \times_{B} A \arrow[r,"\pi_{2}"] \arrow[d,"\pi_{1}"']
& A \arrow[d,"\alpha"]\\
E \arrow[r,"p"']
& B,
\end{tikzcd}
\end{center}
where $p$ is a normal epimorphism in $\C$:
\begin{enumerate}
\item $E \in \mathcal{X}$ and $A \in \mathcal{X}$ implies that $E \times_{B} A \in \mathcal{X}$;
\item $B \in \mathcal{X}$, $E \in \mathcal{X}$ and $E \times_{B} A \in \mathcal{X}$ implies that $A \in \mathcal{X}$.
\end{enumerate}
The notion of \textit{locally semisimple covering} is then defined relatively to a generalized semisimple class $\mathcal X$ in a category $\C$: a morphism $\alpha : A \rightarrow B$ is a locally semisimple covering in $\C$ if there is a normal epimorphism $p : E \rightarrow B$ such that the pullback $p^{*}(\alpha)$ of $\alpha$ along $p$ lies in the corresponding full subcategory $\mathcal X$ of $\C$.

For a fixed $B \in \C$, let $\mathsf{LocSSimple}_{\mathcal{X}}(B)$ be the full subcategory of the slice category $\C \downarrow B$ over $B$ whose objects are pairs $(A,\alpha)$, where $\alpha : A \rightarrow B$ is a locally semisimple covering. 

Under our assumptions, a normal epimorphism $p \colon E \rightarrow B$ in $\C$ induces a category equivalence $K_p \colon \C \downarrow B \rightarrow \mathsf{DiscFib}(Eq(p))$, since $p$ is an effective descent morphism. 
When, moreover, $p \colon E \rightarrow B$ is such that $E$ belongs to $\mathcal X$, the functor $K_p$ restricted to the category of locally semisimple coverings gives an equivalence of categories $$ \mathsf{LocSSimple}_{\mathcal{X}}(B) \cong \mathsf{DiscFib}_{\mathcal{X}}(Eq(p)),$$ where $\mathsf{DiscFib}_{\mathcal{X}}(Eq(p))$ is the full subcategory of $\mathsf{DiscFib}(Eq(p))$ whose objects are the discrete fibrations over $Eq(p)$ (as in \eqref{discrete}) with $F \in \mathcal{X}$:

\begin{theorem} \cite{JMT} \label{thm locally semisimple coverings}
Consider a normal category $\C$ where normal epimorphisms are effective descent morphisms, and $\mathcal{X}$ a generalized semisimple class  in $\C$. If $p : E \rightarrow B$ is a normal epimorphism in $\C$ such that $E \in \mathcal{X}$, there is an equivalence of categories
 \begin{equation}\label{equiv-restricted} 
\mathsf{LocSSimple}_{\mathcal{X}}(B) \cong \mathsf{DiscFib}_{\mathcal{X}}(Eq(p)).
\end{equation} 
\end{theorem}
\begin{proof}
This essentially follows from the two properties of the generalized semisimple classes recalled above, that guarantee that a morphism $f \colon A \rightarrow B$ belongs to the subcategory $ \mathsf{LocSSimple}_{\mathcal{X}}(B)$ if and only if the corresponding discrete fibration \eqref{image-of-K} is such that $E \times_B A \in \mathcal X$ (see \cite{JMT} for the details).
The equivalence 
$K_p \colon \C \downarrow B \rightarrow \mathsf{DiscFib}(Eq(p))$ then (co)restricts to the full subcategories $\mathsf{LocSSimple}_{\mathcal{X}}(B)$ (and $ \mathsf{DiscFib}_{\mathcal{X}}(Eq(p))$), yielding the announced equivalence \eqref{equiv-restricted}.

\end{proof}

\begin{remark}
\emph{Observe that the category $\mathsf{DiscFib}(Eq(p))$ is also called the category of \emph{internal $Eq(p)$-actions} in the literature \cite{JST}.}
\end{remark}

In particular we can consider $\C = \mathsf{PreOrdGrp}$, and $\mathcal{X}$ the class of objects of $\mathsf{ParOrdGrp}$, which is easily seen (by using Lemma $2.7$ in \cite{GRossi}, for instance) to be a generalized semisimple class.
We are therefore in a situation where we can apply Theorem \ref{thm locally semisimple coverings}. We first of all state the following lemma:

\begin{lemma}
A morphism $(h,\bar{h}) \colon (H,P_H) \rightarrow (G,P_G)$ in $\mathsf{PreOrdGrp}$ is a locally semisimple covering (relatively to the subcategory $\mathsf{ParOrdGrp}$) if and only if its kernel is a partially ordered group.
\end{lemma}

\begin{proof}
If $(h,\bar{h})$ is a locally semisimple covering there exists a normal epimorphism $(p,\bar{p}) \colon (E,P_E) \rightarrow (G,P_G)$ such that $(p,\bar{p})^*(h,\bar{h}) \in \mathsf{ParOrdGrp}$. It follows that $\mathsf{Ker} \left( (p,\bar{p})^*(h,\bar{h}) \right) \in \mathsf{ParOrdGrp}$. Now since the diagram
\begin{equation} \label{diag lemma Thm 4.2}
\begin{tikzcd}
(E,P_E) \times_{(G,P_G)} (H,P_H) \arrow[r] \arrow[d,"{(p,\bar{p})^*(h,\bar{h})}"']
& (H,P_H) \arrow[d,"{(h,\bar{h})}"] \\
(E,P_E) \arrow[r,"{(p,\bar{p})}"']
& (G,P_G)
\end{tikzcd}
\end{equation}
is a pullback we have that $\mathsf{Ker}(h,\bar{h}) \cong \mathsf{Ker} \left( (p,\bar{p})^*(h,\bar{h}) \right)$, so that $\mathsf{Ker}(h,\bar{h}) \in \mathsf{ParOrdGrp}$.

Conversely, by Proposition \ref{ParOrdGrp covers PreOrdGrp}, there exists an effective descent morphism (i.e. a normal epimorphism) $(p,\bar{p}) \colon (E,P_E) \rightarrow (G,P_G)$ with $(E,P_E) \in \mathsf{ParOrdGrp}$. Since the diagram \eqref{diag lemma Thm 4.2} is a pullback, we have that $\mathsf{Ker} \left( (p,\bar{p})^*(h,\bar{h}) \right) \cong \mathsf{Ker}(h,\bar{h})$ with $\mathsf{Ker}(h,\bar{h}) \in \mathsf{ParOrdGrp}$ by assumption. Knowing that any torsion-free subcategory is stable by extensions (see \cite{JT} for instance) and that $\mathsf{ParOrdGrp}$ is a torsion-free subcategory of $\mathsf{PreOrdGrp}$ (by Proposition \ref{torsion OrdGrp}), it follows that $(p,\bar{p})^*(h,\bar{h})$ is in $\mathsf{ParOrdGrp}$.
\end{proof}

\begin{remark}
\emph{The previous lemma is a particular case of a more general fact observed in \cite{JMT} (Proposition 2.3) where, more generally, the role of the kernel of an arrow was played by the “\emph{fibers}” (as defined in \cite{JMT}).}
\end{remark}

\begin{theorem} \label{thm action}
Let $(G,P_{G}) \in \mathsf{PreOrdGrp}$. Consider the effective descent morphism $$(f,\bar{f}) : \left(\Z \times G,(\N \times P_{G}) \backslash \{(0,g) \arrowvert g \neq 0 \}\right) \rightarrow (G,P_{G})$$ from Proposition \ref{ParOrdGrp covers PreOrdGrp}. Then there exists an equivalence of categories
\[ {\M}^* \downarrow(G,P_{G}) \cong \mathsf{DiscFib}_{\mathsf{ParOrdGrp}}({Eq(f,\bar{f}))}\]
where ${\M}^* \downarrow(G,P_{G})$  is the category of coverings over $(G,P_{G})$. \end{theorem}

\begin{proof}
Since the morphism $(f,\bar{f}) : \left(\Z \times G,(\N \times P_{G}) \backslash \{(0,g) \arrowvert g \neq 0 \}\right) \rightarrow (G,P_{G})$ from Proposition \ref{ParOrdGrp covers PreOrdGrp} is an effective descent morphism in $\mathsf{PreOrdGrp}$ such that $$\left(\Z \times G,(\N \times P_{G}) \backslash \{(0,g) \arrowvert g \neq 0 \}\right) \in \mathsf{ParOrdGrp}$$ we are allowed to apply Theorem \ref{thm locally semisimple coverings}: there exists then an equivalence of categories
\[\mathsf{LocSSimple}_{\mathsf{ParOrdGrp}}(G,P_{G}) \cong \mathsf{DiscFib}_{\mathsf{ParOrdGrp}}({Eq(f,\bar{f}))}.\]
Thanks to the previous lemma and Theorem \ref{description-coverings} the proof is complete since both the coverings and the locally semisimple coverings (over $(G,P_G)$) are described as the arrows $(h, \overline{h})  \colon (H,P_{H}) \rightarrow (G,P_{G})$ with $\mathsf{Ker}(h,\overline{h}) \in \mathsf{ParOrdGrp}$.
\end{proof}

Note that the internal equivalence relation $Eq(f,\bar{f})$ from Theorem \ref{thm action} is in fact the \textit{Galois groupoid} $\mathsf{Gal}(f,\bar{f})$ associated with the effective descent morphism $(f,\bar{f})$ \cite{Janelidze}. By definition the Galois groupoid associated with $(f,\bar{f})$ is indeed the image of $Eq(f,\bar{f})$ by the reflector $F : \mathsf{PreOrdGrp} \rightarrow \mathsf{ParOrdGrp}$. But since the diagram
\begin{center}
\begin{tikzcd}
(Eq(f),Eq(\bar{f})) \arrow[r,shift left=2.3] \arrow[r,shift right=2.3]
& (E,P_{E}) \arrow[l]
\end{tikzcd}
\end{center} 
lies in $\mathsf{ParOrdGrp}$ (where we write $(E,P_{E})$ for $\left(\Z \times G,(\N \times P_{G}) \backslash \{(0,g) \arrowvert g \neq 0 \}\right)$) the image of $Eq(f,\bar{f})$ by the reflector $F$ is $Eq(f,\bar{f})$ itself. In other words $Eq(f,\bar{f})$ is $\mathsf{Gal}(f,\bar{f})$, and 

\[{\M}^*\downarrow (G,P_{G}) \cong \mathsf{DiscFib}_{\mathsf{ParOrdGrp}}(\mathsf{Gal}(f,\bar{f}))
.\]
This equivalence is the \emph{classification} of the coverings as internal $\mathsf{Gal}(f,\bar{f})$-actions.

\begin{remark}
\emph{Besides its interest for the classification of coverings in the category of preordered groups, the above result also provides an example of application of Theorem 3.1 in \cite{JMT} in a non-exact setting (see Remark 3.2 (e) in \cite{JMT}).}
\end{remark}


\section{The torsion subcategory of protomodular objects}

In this last section we show that the reflection \ref{main-adj} gives also rise to a pretorsion theory in the category $\mathsf{PreOrdGrp}$ of preordered groups.
\subsection*{Pretorsion theories in general categories}

The concept of \textit{pretorsion theory} \cite{FF} allows one to extend the notion of torsion theory to a non-pointed category. Here we only recall the basic results which will be useful for this work, and we refer to \cite{Facchini-Finocchiaro-Gran} for the fundamental aspects of the theory. To adapt Definition \ref{def torsion theory} to a general category (not necessarily pointed)  a (non-empty) class $\z$ of objects of $\C$ is introduced, which somehow plays the role of the zero object, and we denote by $\n$ the class of morphisms in $\C$ that factorize through an object of $\z$. These special morphisms are called \textit{$\z$-trivial}. One can then extend the notions of kernel, cokernel and short exact sequence to get the notions of \textit{$\z$-prekernel}, \textit{$\z$-precokernel} and \textit{short $\z$-preexact sequence}.

From now on we assume $\C$ to be an arbitrary category. Given an arrow $f : A \rightarrow B$ in $\C$, one says that $k : K \rightarrow A$ is a \emph{$\z$-prekernel} of $f$ when
\begin{itemize}
\item[$\bullet$] $f \cdot k \in \n$;
\item[$\bullet$] for any morphism $\alpha : X \rightarrow A$ such that $f \cdot \alpha \in \n$, there exists a unique arrow $\phi : X \rightarrow K$ such that $k \cdot \phi = \alpha$.
\end{itemize}
Dually, an arrow $c : B \rightarrow C$ is a \emph{$\z$-precokernel} of $f \colon A \rightarrow B$ when
\begin{itemize}
\item[$\bullet$] $c \cdot f \in \n$;
\item[$\bullet$] for any morphism $\alpha : B \rightarrow X$ such that $\alpha \cdot f \in \n$, there exists a unique arrow $\phi : C \rightarrow X$ such that $\phi \cdot c = \alpha$.
\end{itemize}

Any $\z$-prekernel is a monomorphism and, dually, any $\z$-precokernel is an epimorphism.

\begin{definition}
Let $f : A \rightarrow B$ and $g : B \rightarrow C$ be two arrows in $\C$. The sequence
\begin{center}
\begin{tikzcd}
0 \arrow[r]
& A \arrow[r,"f"]
& B \arrow[r,"g"]
& C \arrow[r]
& 0
\end{tikzcd}
\end{center}
is a \emph{short $\z$-preexact sequence} when $f$ is a $\z$-prekernel of $g$ and $g$ is a $\z$-precokernel of $f$.
\end{definition}

We are now ready to recall the definition of \textit{pretorsion theory} \cite{FF, Facchini-Finocchiaro-Gran} (see also \cite{Mantovani, GJ} for an interesting and closely approach based on the notion of ideal of morphisms): 
\begin{definition} \label{def pretorsion theory}
A \emph{$\z$-pretorsion theory} in the category $\C$ is given by a pair $(\T,\F)$ of full replete subcategories of $\C$, with $\z = \T \cap \F$, such that:
\begin{itemize}
\item[$\bullet$] any morphism in $\C$ from $T \in \T$ to $F \in \F$ belongs to $\n$;
\item[$\bullet$] for any object $C$ of $\C$ there exists a short $\z$-preexact sequence
\begin{center}
\begin{tikzcd}
0 \arrow[r]
& T \arrow[r,"\epsilon_{C}"]
& C \arrow[r,"\eta_{C}"]
& F \arrow[r]
& 0
\end{tikzcd}
\end{center}
with $T \in \T$ and $F \in \F$.
\end{itemize}
\end{definition}

In a similar way as for classical torsion theories, the torsion-free subcategory $\F$ of a pretorsion theory $(\T,\F)$ is epireflective in $\C$ and, dually, the torsion subcategory $\T$ is monocoreflective in $\C$ \cite{Facchini-Finocchiaro-Gran}.

Also in this more general situation the $C$-component of the unit $\eta$ of the adjunction relative to the reflector $F \colon \C \rightarrow \F$ is given by the arrow $\eta_{C} : C \rightarrow F = UF(C)$ of Definition \ref{def pretorsion theory} where $U : \F \rightarrow \C$ is the inclusion functor, and the $C$-component of the counit of the adjunction $V \dashv T$ is given by the arrow $\epsilon_{C} : T = VT(C) \rightarrow C$, where $V : \T \rightarrow \C$ stands for the inclusion functor.

\subsection*{A pretorsion theory in the category $\mathsf{PreOrdGrp}$ of preordered groups}

\begin{proposition} \label{pretorsion OrdGrp}
The pair $(\mathsf{ProtoPreOrdGrp},\mathsf{ParOrdGrp})$ of full replete subcategories of $\mathsf{PreOrdGrp}$ is a $\z$-pretorsion theory in $\mathsf{PreOrdGrp}$, where $\z = \mathsf{ProtoPreOrdGrp} \cap \mathsf{ParOrdGrp}$ is given by 
$$\z = \{ (G,P_{G}) \, \mid \,  P_{G} = 0 \}.$$
Observe that the preordered groups in $\z$ are the ones endowed with the discrete order.
\end{proposition}

\begin{proof}
To prove that any arrow $(f,\bar{f}) : (G,P_{G}) \rightarrow (H,P_{H})$ in $\mathsf{PreOrdGrp}$, with $(G,P_{G}) \in \mathsf{ProtoPreOrdGrp}$ and $(H,P_{H}) \in \mathsf{ParOrdGrp}$ factorizes through an object in $\z$, first observe that the following diagram commutes
\begin{equation} \label{diagram pretorsion theory}
\begin{tikzcd} [column sep = small, row sep = small] 
P_{G} \arrow[rr,"\bar{f}"] \arrow[dr,dotted] \arrow[ddd,tail]
& & P_{H} \arrow[ddd,tail]\\
 & 0 \arrow[d,tail] \arrow[ur,dotted]
& \\
 & f(G) \arrow[dr,dotted,tail]
& \\
G \arrow[ur,dotted,two heads] \arrow[rr,"f"']
& & H,
\end{tikzcd}
\end{equation}
where $f(G)$ is the image of the group morphism $f \colon G \rightarrow H$. Indeed, as explained in the first part of the proof of Proposition \ref{torsion OrdGrp}, any monoid morphism from a group to a reduced monoid is the $0$-arrow.
Since $f(G)$ is a group and since any part of the diagram \ref{diagram pretorsion theory} commutes, we conclude that any arrow in $\mathsf{PreOrdGrp}$ from an object of $\mathsf{ProtoPreOrdGrp}$ to an object of $\mathsf{ParOrdGrp}$ factorizes through an object of $\z$, i.e it belongs to $\n$.

Consider now any preordered group $(G,P_{G})$. We will work as before with the normal subgroup $N_{G}$ of $G$. We then consider the (regular epimorphism, monomorphism)-factorization of $\eta_{G} \cdot g$ in the category $\mathsf{Mon}$ of monoids, where
\begin{tikzcd}
G \arrow[r,two heads,"\eta_{G}"] 
& G/N_{G}
\end{tikzcd}
is the quotient morphism and
\begin{tikzcd}
P_{G} \arrow[r,tail,"g"]
& G
\end{tikzcd}
is the inclusion arrow:
\begin{center}
\begin{tikzcd}
P_{G} \arrow[r,two heads,"\bar{\eta}_{G}"] \arrow[d,tail,"g"']
& P_{G}/N_{G} \arrow[d,tail,"\psi_{G}"]\\
G \arrow[r,two heads,"\eta_{G}"']
& G/N_{G}.
\end{tikzcd}
\end{center}
Let us prove that the sequence
\begin{center}
\begin{tikzcd}
N_{G} \arrow[r,tail,"i"] \arrow[d,tail,"\phi_{G}"']
& P_{G} \arrow[r,two heads,"\bar{\eta}_{G}"] \arrow[d,tail,"g"']
& P_{G}/N_{G} \arrow[d,tail,"\psi_{G}"]\\
G \arrow[r,equal]
& G \arrow[r,two heads,"\eta_{G}"']
& G/N_{G} 
\end{tikzcd}
\end{center}
in which $(G,N_{G}) \in \mathsf{ProtoPreOrdGrp}$ and $(G/N_{G},P_{G}/N_{G}) \in \mathsf{ParOrdGrp}$ is a short $\z$-preexact sequence in $\mathsf{PreOrdGrp}$. 

We begin by showing that $(\eta_{G},\bar{\eta}_{G})$ is the $\z$-precokernel of the arrow $(1_{G},i)$. Let us consider a morphism $(f,\bar{f}) : (G,P_{G}) \rightarrow (H,P_{H})$ in $\mathsf{PreOrdGrp}$ such that $(f,\bar{f}) \cdot (1_{G},i) \in \n$, i.e. such that $(f,\bar{f}) \cdot (1_{G},i)$ factorizes through an object $(A,0)$ of $\z$: $(f,\bar{f}) \cdot (1_{G},i) = (b,\bar{b}) \cdot (a,\bar{a})$.
\begin{center}
\begin{tikzcd}
N_{G} \arrow[rr,tail,"i"] \arrow[dr,"\bar{a}"] \arrow[ddd,tail,"\phi_{G}"']
& & P_{G} \arrow[rr,two heads,"\bar{\eta}_{G}"] \arrow[dr,"\bar{f}"] \arrow[ddd,tail,"g"']
& & P_{G}/N_{G} \arrow[ddd,tail,"\psi_{G}"] \arrow[dl,dotted,"\bar{\alpha}"']\\
 & 0 \arrow[rr,near start,"\bar{b}"] \arrow[d,tail]
& & P_{H} \arrow[d,tail,"h"]
& \\
 & A \arrow[rr,near start,"b"']
& & H & \\
G \arrow[rr,equal] \arrow[ur,"a"']
& & G \arrow[ur,"f"'] \arrow[rr,two heads,"\eta_{G}"']
& & G/N_{G} \arrow[ul,dotted,"\alpha"]
\end{tikzcd}
\end{center} 
In particular $f \cdot \phi_{G} = b \cdot a \cdot \phi_{G} = 0$. Since $\eta_{G}$ is the cokernel of $\phi_{G}$ in the category $\mathsf{Grp}$ of groups, by the universal property of the cokernel, there exists a unique arrow $\alpha : G/N_{G} \rightarrow H$ in $\mathsf{Grp}$ such that $\alpha \cdot \eta_{G} = f$. Now, seeing that $\bar{\eta}_{G}$ is a regular epimorphism in $\mathsf{Mon}$ and that $h$ is a monomorphism, the universal property of strong epimorphisms yields a unique arrow $\bar{\alpha} : P_{G}/N_{G} \rightarrow P_{H}$ such that $\bar{\alpha} \cdot \bar{\eta}_{G} = \bar{f}$ and $h \cdot \bar{\alpha} = \alpha \cdot \psi_{G}$. In other words there exists a unique arrow $(\alpha,\bar{\alpha}) : (G/N_{G},P_{G}/N_{G}) \rightarrow (H,P_{H})$ in $\mathsf{PreOrdGrp}$ such that $(\alpha,\bar{\alpha}) \cdot (\eta_{G},\bar{\eta}_{G}) = (f,\bar{f})$, i.e. $(\eta_{G},\bar{\eta}_{G})$ is the $\z$-precokernel of $(1_{G},i)$.

Next we show that $(1_{G},i)$ is the $\z$-prekernel of $(\eta_{G},\bar{\eta}_{G})$. Consider $(f,\bar{f}) : (H,P_{H}) \rightarrow (G,P_{G})$ in $\mathsf{PreOrdGrp}$ such that $(\eta_{G},\bar{\eta}_{G}) \cdot (f,\bar{f}) \in \n$, i.e. $(\eta_{G},\bar{\eta}_{G}) \cdot (f,\bar{f})$ factorizes through an object $(A,0)$ of $\z$: $(\eta_{G},\bar{\eta}_{G}) \cdot (f,\bar{f}) = (b,\bar{b}) \cdot (a,\bar{a})$.
\begin{center}
\begin{tikzcd}
N_{G} \arrow[rr,tail,"i"] \arrow[ddd,tail,"\phi_{G}"']
& & P_{G} \arrow[rr,two heads,"\bar{\eta}_{G}"] \arrow[ddd,tail,"g"']
& & P_{G}/N_{G} \arrow[ddd,tail,"\psi_{G}"]\\
 & P_{H} \arrow[ul,dotted,"\bar{\alpha}"'] \arrow[ur,"\bar{f}"] \arrow[d,tail,"h"'] \arrow[rr,near end,"\bar{a}"]
& & 0 \arrow[ur,"\bar{b}"] \arrow[d,tail]
& \\
 & H \arrow[dl,dotted,"\alpha"] \arrow[dr,"f"'] \arrow[rr,near end,"a"']
& & A \arrow[dr,"b"']
& \\
G \arrow[rr,equal]
& & G \arrow[rr,two heads,"\eta_{G}"']
& & G/N_{G}
\end{tikzcd}
\end{center}
Let us take $\alpha = f$, since this is the only possible arrow such that $1_G \cdot \alpha =f$. Now we have that $\bar{\eta}_{G} \cdot \bar{f} = 0$, hence $\eta_{G} \cdot g \cdot \bar{f} = 0$. Since $\phi_{G}$ is the kernel of $\eta_{G}$ in $\mathsf{Mon}$ (and in $\mathsf{Grp}$) there is a unique arrow $\bar{\alpha} : P_{H} \rightarrow N_{G}$ such that $\phi_{G} \cdot \bar{\alpha} = g \cdot \bar{f}$. Then
\[g \cdot i \cdot \bar{\alpha} = \phi_{G} \cdot \bar{\alpha} = g \cdot \bar{f}\]
and since $g$ is a monomorphism it follows that $i \cdot \bar{\alpha} = \bar{f}$. The morphism $\bar{\alpha}$ is moreover unique with this property, because $i$ is a monomorphism. As a conclusion $(\alpha,\bar{\alpha}) : (H,P_{H}) \rightarrow (G,N_{G})$
is the unique arrow such that $(1_{G},i) \cdot (\alpha,\bar{\alpha}) = (f,\bar{f})$, and $(1_{G},i)$ is the $\z$-prekernel of $(\eta_{G},\bar{\eta}_{G})$. 
\end{proof}

From this Proposition we deduce in particular that the subcategory $\mathsf{ProtoPreOrdGrp}$ of protomodular objects is monocoreflective in $\mathsf{PreOrdGrp}$: 
\begin{corollary}
The subcategory $\mathsf{ProtoPreOrdGrp}$ of protomodular objects is monocoreflective in the category $\mathsf{PreOrdGrp}$ of preordered groups:
\begin{equation}\label{second adj}
\begin{tikzcd}[column sep = tiny]
\mathsf{PreOrdGrp} \arrow[rr,shift right=0.2cm,"T"']
& \bot
& \mathsf{ProtoPreOrdGrp}. \arrow[ll,shift right=0.2cm,hook,"V"']
\end{tikzcd}
\end{equation}
\end{corollary}
It turns out that the inclusion functor $V : \mathsf{ProtoPreOrdGrp} \hookrightarrow \mathsf{PreOrdGrp}$ is not only a left adjoint but also a right adjoint:
\begin{proposition}\label{reflective}
The functor $V : \mathsf{ProtoPreOrdGrp} \hookrightarrow \mathsf{PreOrdGrp}$ has a left adjoint functor $E : \mathsf{PreOrdGrp} \rightarrow \mathsf{ProtoPreOrdGrp}$:
\begin{equation}\label{third adj}
\begin{tikzcd}[column sep = tiny]
\mathsf{PreOrdGrp} \arrow[rr,shift left=0.2cm,"E"]
& \bot
& \mathsf{ProtoPreOrdGrp}. \arrow[ll,shift left=0.2cm,hook',"V"]
\end{tikzcd}
\end{equation}
\end{proposition}
\begin{proof}
We begin with the construction of the functor $E : \mathsf{PreOrdGrp} \rightarrow \mathsf{ProtoPreOrdGrp}$. Let $(G,P_{G})$ be any preordered group. Consider the subgroup $M_{G}$ of $G$ generated by all elements in $P_{G} \cup (- P_{G})$, where we write $-P_{G}$ for the submonoid
\[- P_{G} = \{ x \in G \ \arrowvert \ \exists \ g \in P_{G} \ \text{such that} \ x = -g\}.\]
Any element $m$ of $M_{G}$ is of the form
$m = g_{1} - g_{2} + g_{3} - \dots + g_{n-1} - g_{n}$
for $g_{1}, \dots, g_{n} \in P_{G}$. Since both $P_{G}$ and $- P_{G}$ are submonoids of $G$ it is clear that $M_{G}$ is a subgroup of $G$. And this subgroup is in addition normal in $G$. Indeed, let $g \in G$ and let $m = g_{1} - g_{2} + \dots + g_{n-1} - g_{n}$ be an element in $M_{G}$ (where $g_{1}, \dots, g_{n} \in P_{G}$). Then
\begin{align*}
g + m - g & = g + g_{1} - g_{2} + \dots + g_{n-1} - g_{n} - g\\
& = (g + g_{1} - g) + (g - g_{2} - g) + \dots + (g + g_{n-1} - g) + (g - g_{n} -g)\\
& = (g + g_{1} - g) - (g + g_{2} -g) + \dots + (g + g_{n-1} -g) - (g + g_{n} - g)
\end{align*}
with $g + g_{i} - g \in P_{G}$ for any $i \in \{1, \dots ,n\}$ since $P_{G}$ is closed under conjugation in $G$, and $g + m - g \in M_{G}$. Accordingly $(G,M_{G})$ in an object of the subcategory $\mathsf{ProtoPreOrdGrp}$. This construction is obviously functorial,  and we write $E \colon  \mathsf{PreOrdGrp} \rightarrow \mathsf{ProtoPreOrdGrp}$ for the functor defined on objects by $E(G,P_{G}) = (G,M_{G})$, for any $(G,P_{G}) \in \mathsf{PreOrdGrp}$.

Let us now prove that this functor $E$ is a left adjoint of the functor $V \colon \mathsf{ProtoPreOrdGrp} \rightarrow \mathsf{PreOrdGrp}$. Let $(G,P_{G}) \in \mathsf{PreOrdGrp}$, and let us check that the $(G,P_{G})$-component of the unit of the adjunction is given by the arrow $(1_{G},j)$
\begin{center}
\begin{tikzcd}
P_{G} \arrow[r,tail,"j"] \arrow[d,tail,"g"']
& M_{G} \arrow[d,tail,"n"]\\
G \arrow[r,equal]
& G
\end{tikzcd}
\end{center}
where 
\begin{tikzcd}
P_{G} \arrow[r,tail,"j"]
& M_{G}
\end{tikzcd}
is the inclusion morphism. Let $(H,P_{H}) \in \mathsf{ProtoPreOrdGrp}$ and consider any morphism $(f,\bar{f}) : (G,P_{G}) \rightarrow V(H,P_{H}) = (H,P_{H})$.
\begin{center}
\begin{tikzcd}
P_{G} \arrow[rr,tail,"j"] \arrow[dr,"\bar{f}"'] \arrow[ddd,tail,"g"']
& & M_{G} \arrow[ddd,tail,"n"] \arrow[dl,dotted,"\bar{\phi}"]\\
 & P_{H} \arrow[d,tail,"h"] & \\
 & H & \\
G \arrow[ur,"f"] \arrow[rr,equal]
& & G \arrow[ul,dotted,"\phi"']
\end{tikzcd}
\end{center}
There exists a unique morphism $\phi = f : G \rightarrow H$ with the property $\phi \cdot 1_{G} = f$. We then observe that, for any $m = g_{1} - g_{2} + \dots + g_{n-1} - g_{n} \in M_{G}$ (with $g_{1}, \dots, g_{n} \in P_{G}$),
\begin{align*}
f(m) & = f(g_{1} - g_{2} + \dots + g_{n-1} - g_{n})\\
& = f(g_{1}) - f(g_{2}) + \dots + f(g_{n-1}) - f(g_{n})
\end{align*}
with $f(g_{i}) \in P_{H}$ since $g_{i} \in P_{G}$ for any $i \in \{1, \dots , n\}$. Hence $f(m) \in P_{H}$ since $P_{H}$ is a group by assumption. This means that the restriction $f_{\arrowvert M_{G}}$ of $f$ to $M_{G}$ takes its values in $P_{H}$. Let us then define $\bar{\phi} = f_{\arrowvert M_{G}} : M_{G} \rightarrow P_{H}$. We can then check that $\bar{\phi} \cdot j = \bar{f}$, and we observe that, for any $m \in M_{G}$,
\[(\phi \cdot n) (m) = (f \cdot n) (m) = f(m) = \bar{\phi}(m) = (h \cdot \bar{\phi})(m),\]
so that $\phi \cdot n = h \cdot \bar{\phi}$. Accordingly there exists a unique morphism $(\phi,\bar{\phi}) = (f, f_{\arrowvert M_{G}}) : (G,M_{G}) \rightarrow (H,P_{H})$ such that $(\phi,\bar{\phi}) \cdot (1_{G},j) = (f,\bar{f})$, and the proof is complete.
\end{proof}


\end{document}